\documentclass{scrartcl}
\usepackage{amssymb,amsmath,amsthm}
\usepackage{dsfont}
\usepackage{float}
\usepackage{color}
\usepackage{url}
\usepackage{graphicx}

\usepackage{tikz}
\usetikzlibrary{automata,positioning,arrows}
\usepackage{authblk}
\newtheorem{Theorem}{Theorem}[section] 
\newtheorem{Lemma}[Theorem]{Lemma}     

\newtheorem{Definition}[Theorem]{Definition}
\newtheorem{Proposition}[Theorem]{Proposition}
\newtheorem{Conjecture}[Theorem]{Conjecture}

\theoremstyle{definition}
\newtheorem{Example}[]{Example}
\newtheorem{Remark}[Theorem]{Remark}

\usepackage[square, sort, comma, numbers]{natbib} 

\renewcommand{\tilde}{\widetilde}

\newcommand{\GG}{\mathbb{G}}

\newcommand{\hQ}{\hat{\MQ}}

\newcommand{\FP}{\mathfrak{p}}

\newcommand{\MA}{\mathds{A}}

\newcommand{\MN}{\mathds{N}}

\newcommand{\MZ}{\mathds{Z}}
\newcommand{\MQ}{\mathds{Q}}
\newcommand{\MU}{\mathds{U}}
\newcommand{\MR}{\mathds{R}}
\newcommand{\MC}{\mathds{C}}

\newcommand{\MT}{\mathds{T}}

\newcommand{\MO}{\mathds{O}}
\newcommand{\CO}{\mathcal{O}}

\newcommand{\CH}{\mathcal{H}}
\newcommand{\Ct}{\vartheta}
\newcommand{\ind}{\mathds{1}}

\newcommand{\fp}{\mathfrak{p}}

\newcommand{\fq}{\mathfrak{q}}

\newcommand{\genus}{\mathrm{genus}}

\newcommand{\GL}{\mathrm{GL}}
\newcommand{\SL}{\mathrm{SL}}

\newcommand{\GSp}{\mathrm{GSp}}

\newcommand{\stab}{\operatorname{Stab}} \newcommand{\diag}{\operatorname{diag}}
 
\newcommand{\C}{\mathcal{C}} \renewcommand{\c}{\mathfrak{c}}

\newcommand{\End}{\mathrm{End}}
\newcommand{\Aut}{\mathrm{Aut}}

\newcommand{\Sym}{\mathrm{Sym}}

\newcommand{\tr}{\mathrm{Tr}}

\newcommand{\ord}{\mathrm{ord}}

\title{Automorphic Forms on Feit's Hermitian Lattices}
\author[1]{Neil Dummigan}
 \author[2]{ Sebastian Sch\"onnenbeck\thanks{The author is supported by the DFG collaborative research center TRR 195} }

\affil[1]{School of Mathematics and Statistics \\
University of Sheffield \\
Hicks Building \\
Hounsfield Road \\
Sheffield, S3 7RH \\
United Kingdom \\
\texttt{n.p.dummigan@sheffield.ac.uk}}

\affil[2]{ RWTH Aachen University\\
   Lehrstuhl D f\"ur Mathematik\\
   Pontdriesch 14/16, 52062 Aachen\\
   Germany\\
  \texttt{sebastian.schoennenbeck@rwth-aachen.de}
}

\begin{document}
\maketitle

\begin{abstract}
We consider the genus of $20$ classes of unimodular Hermitian lattices of rank $12$ over the Eisenstein integers. This set is the domain for a certain space of algebraic modular forms. We find a basis of Hecke eigenforms, and guess global Arthur parameters for the associated automorphic representations, which recover the computed Hecke eigenvalues. Congruences between Hecke eigenspaces, combined with the assumed parameters, recover known congruences for classical modular forms, and support new instances of conjectured Eisenstein congruences for $\MU(2,2)$ automorphic forms.
\end{abstract}

\section{Introduction}
Nebe and Venkov \cite{NebeVenkov} looked at formal linear combinations of the $24$ Niemeier lattices, which represent classes in the genus of even, unimodular, Euclidean lattices of rank $24$. They found a set of $24$ eigenvectors for the action of an adjacency operator for Kneser $2$-neighbours, with distinct integer eigenvalues. What they did was equivalent to computing a set of Hecke eigenforms in a space of scalar-valued modular forms for a definite orthogonal group $\MO_{24}$. They made, and proved most of, a conjecture on the degrees in which the Siegel theta series of these eigenvectors are first non-vanishing.

Chenevier and Lannes \cite{ChenevierLannes} reconsidered the results of Nebe and Venkov, in the light of work of Arthur \cite{Arthur}, and found (with proof) the endoscopic type of the automorphic representation of $\MO_{24}(\MA_{\MQ})$ generated by each eigenvector. Each of these automorphic representations is a ``lift'' built out of automorphic representations of smaller rank groups, related to elliptic modular forms, and certain vector-valued Siegel modular forms of genus $2$. They looked at various easily-proved congruences of Hecke eigenvalues between pairs of eigenvectors. Writing the eigenvalues in terms of the endoscopic decompositions, they obtained, after much cancellation from both sides, not only well-known congruences such as Ramanujan's $\tau(p)\equiv 1+p^{11}\pmod{691}$, but also the first proved instance of a conjecture of Harder on congruences between Hecke eigenvalues of vector-valued Siegel cusp forms of genus $2$ and cusp forms of genus $1$, modulo large primes occurring in critical values of the $L$-functions of the latter \cite{Harder123}. The same method has subsequently been employed by M\'egarban\'e to prove several similar congruences, and also some involving automorphic forms for $\mathrm{SO}_7$ \cite{Megarbane}.

In this paper we replace the Niemeier lattices by the genus of $20$ classes of unimodular Hermitian lattices of rank $12$ over the Eisenstein integers. Thus the orthogonal group $\MO_{24}$ is replaced by a definite unitary group $\MU_{12}$. These classes were enumerated, and given explicit representatives, by Feit \cite{FeitSomeLattices}. Using Kneser $\FP$-neighbours, in particular for $\FP=(2)$ (but also for $\FP=(\sqrt{-3})$), we obtain a basis of $20$ eigenvectors in a space of scalar-valued algebraic modular forms. Using the computed Hecke eigenvalues, in tandem with the clues provided by various congruences between pairs of eigenvectors, we make compelling guesses for the endoscopic type of the automorphic representation of $\MU_{12}(\MA_{\MQ})$ generated by each eigenvector.

Assuming these guesses, after cancellation we more-or-less recover various congruences involving elliptic modular forms of levels $1$ or $3$, including Ramanujan's congruence, Eisenstein congruences ``of local origin'' \cite{DumFret} and Ribet-Diamond level-raising congruences \cite{Ribet,Diamond}. But we also obtain instances of conjectured congruences involving automorphic representations of $\MU(2,2)$, analogous to those of Harder, again with the moduli coming from critical $L$-values \cite{DumUnitary}. Indeed, the motivation for this work was to prove such congruences, following the work of Chenevier and Lannes on Harder's conjecture. However, because there is now a ``bad'' prime $3$, it appears that it is not yet technically feasible to do something similar here, see Remark \ref{technical}. So perhaps the main ``result'' of the paper is the table at the beginning of \S 4, with its computed Hecke eigenvalues and conjectured global Arthur parameters.

In Section 2 we review some background on algebraic modular forms. In Section 3 we describe how to use $\FP$-neighbours of Hermitian lattices to compute Hecke operators for definite unitary groups, giving us the matrices for $T_{(2)}$ and $T_{(\sqrt{-3})}$ on the $20$-dimensional space of primary interest in this paper. Section 4 starts with a table of Hecke eigenvalues and conjectured global Arthur parameters, and proceeds to show how to recover the former from the latter (Propositions \ref{recoverT2},\ref{RecoverT3}). In Section 5 we prove congruences of Hecke eigenvalues between pairs of eigenspaces, and show how they may be explained using the conjectured global Arthur parameters (providing further evidence for the latter), then concentrating in Section 6 on congruences involving $\MU(2,2)$. Section 7 contains some guesses on Hermitian theta series and Hermitian Ikeda lifts. It should be noted that since the work of Feit, unimodular Hermitian lattices of ranks $13,14$ and $15$ over the Eisenstein integers have also been classified, by Abdukhalikov and Scharlau \cite{Abdukhalikov13, Abdukhalikov1415}, the total numbers of classes being $34, 93$ and $353$, respectively.
\section{Preliminaries and notation}\label{Prelim}
We establish the notation for the rest of the article and provide the necessary background on algebraic modular forms.
Let $E$ be an imaginary quadratic number field and $n \in \MN$. We denote by $V_n$ the $n$-dimensional $E$-space $E^n$ endowed with the standard Hermitian form
\begin{equation}
 \langle v,w \rangle = v^\dagger w,
\end{equation}
where $v^\dagger$ denotes the conjugate transpose of $v$.
By $\MU_n$ we will denote the unitary group stabilising this form. In other words $\MU_n$ is the linear algebraic group over $\MQ$ whose group of $A$-rational points is given by
\begin{equation}
 \MU_n(A)=\{g \in \GL_n(A \otimes_\MQ E)~|~g^\dagger g = I_n \}
\end{equation}
for any commutative $\MQ$-algebra $A$, where $I_n$ is the $n \times n$-unit matrix.

The group $\MU_n(\MR)$ is the usual unitary group of degree $n$ over the complex numbers and hence compact. Finally we let
\begin{equation}
 \mathrm{U}_n:=\MU_n(\MQ)=\{g \in \GL_n(E)~|~g^\dagger g = I_n\}.
\end{equation}

\subsection{Hermitian lattices}
Let $\CO_E$ be the ring of integers of $E$. By an Hermitian lattice we will always mean a full $\CO_E$-lattice in $V_n$.
\begin{Definition}
Let $L \subset V_n$ be an Hermitian lattice. The dual of $L$ is defined as
\begin{equation}
 L^\# :=\{v \in V_n~|~\langle v,L \rangle \subset \CO_E \}.
\end{equation}
Let $I_1,...,I_n$ be the invariant factors of $L^\#$ and $L$. (These fractional ideals of $E$ are as defined on \cite[p.14]{GreenbergVoightLatticeMethods}.) Then
\begin{equation}
 \partial(L):=\prod_{i=1}^n I_i
\end{equation}
is called the discriminant of $L$. We call $L$ integral if $L \subset L^\#$. For a fractional ideal $I \subset E$ we call $L$ $I$-modular if $ L^\# = I \cdot L$. If $L$ is $\CO_E$-modular (so $L=L^\#$) we call $L$ unimodular.
\end{Definition}

\begin{Definition}
Let $L$ be an Hermitian lattice. The group
\begin{equation}
\Aut(L):=\stab_{\mathrm{U}_n}(L)=\{g \in \mathrm{U}_n~|~gL=L \}
\end{equation}
is called the automorphism group of $L$.
\end{Definition}

\subsection{Open compact subgroups arising from lattices}

Let $\hQ=\hat{\MZ}\otimes_{\MZ}\MQ$ be the $\MQ$-algebra of finite adeles of $\MQ$. So
\begin{equation}
 \hat{\MQ}=\left\{(x_p)_p \in \prod_p\MQ_p~|~x_p \in \MZ_p \text{ f.a.a. } p \right\}.
\end{equation}
Let $\MA_\MQ:=\MR \times \hQ$ be the full ring of adeles. The formula $gL:=(g(L\otimes_{\MZ}\hQ))\cap V_n$ gives an action of $\MU_n(\hQ)$ on the set of Hermitian lattices in $V_n$.
The orbit of $L$ is called the $\MU_n$-genus of $L$, denoted $\genus(L)$. If $L$ is integral, so is anything in $\genus(L)$.
Given an Hermitian lattice $L$ in $V_n$, we obtain an open, compact subgroup $K_L=\stab_{\MU_n(\hQ)}(L)$ with
\begin{equation}
 K_L=\prod_{p}K_{L,p},\text{ where } K_{L,p}=\stab_{\MU_n(\MQ_p)}(L \otimes \MZ_p).
\end{equation}
The group $K_{L,p}$ is a hyperspecial maximal compact subgroup for all but finitely many finite primes (cf. \cite[Prop. 3.3]{CohenNebePlesken}). Moreover, two open compact subgroups $K$ and $K'$ arising as stabilizers in this way coincide at all but finitely many places.

In this situation, decomposing $\MU_n(\hQ)$ into $\mathrm{U}_n$-$K_L$-double cosets amounts to the same as finding representatives for the isomorphism classes in the $\MU_n$-genus of $L$, i.e. decomposing the $\MU_n(\hQ)$-orbit of $L$ into $\mathrm{U}_n$-orbits. The class number $|\Sigma_{K_L}|$, where $\Sigma_{K_L}:=\mathrm{U}_n \backslash \MU_n(\hQ) / K_L$, is then also called the class number of $L$.

\begin{Remark}
 The lattice $L_0:=\CO_E^n$ is unimodular and its $\MU_n$-genus consists exactly of all unimodular lattices. Moreover, for $E=\MQ(\sqrt{-3})$ the unimodular lattices in dimension at most $12$ are fully classified, see \cite{FeitSomeLattices}.
\end{Remark}

\subsection{Algebraic modular forms}
The primary reference for this subsection is Gross's original article \cite{GrossAlgebraicModularForms}. In addition we also refer to the more algorithmically oriented article by Greenberg and Voight \cite{GreenbergVoightLatticeMethods}.

Let $\rho:\MU_n \rightarrow
\GL_V$ be an irreducible finite-dimensional rational representation of $\MU_n$ defined over $\MC$, and
 let $K$ be an open compact subgroup of $\MU_n(\hQ)$. For us, $K$ will always be a lattice stabiliser $K_L$ as above, and $\rho$ will usually be the trivial representation.

\begin{Definition}
  The space of algebraic modular forms of weight $V$ and level $K$ is defined as
\begin{equation}
\begin{split}
M(V,K)&=\left\{ f:\MU_n(\hat{Q})/ K \rightarrow V ~|~  \substack{f(g\gamma )=gf(\gamma) \text{ for } \\ \gamma \in \MU_n(\hat{Q}),~g \in \mathrm{U}_n }\right\} \\
&\cong \left\{ f:\MU_n(\hat{Q}) \rightarrow V ~|~  \substack{f(g \gamma \kappa)=gf(\gamma) \text{ for } \gamma \in
\MU_n(\hat{Q}),\\~g \in \mathrm{U}_n,\kappa \in K  }\right\}.
\end{split}
\end{equation}
\end{Definition}

 The structure of $M(V,K)$ is summarized in the following proposition.
\begin{Proposition}[\protect{\cite[Prop. (4.3),(4.5)]{GrossAlgebraicModularForms}}]
 Let $\Sigma_K:=\mathrm{U}_n \backslash \MU_n(\hQ) / K$. The following holds:
\begin{enumerate}
 \item The set $\Sigma_K$ is finite.
 \item If $\alpha_i,1\leq i \leq h,$ is a system of representatives for $\Sigma_K$ and
\begin{equation}
 \Gamma_i:=\mathrm{U}_n \cap \alpha_i K \alpha_i^{-1},
\end{equation}
then
\begin{equation}\label{ModularFormSpaceDecomp}
 M(V,K) \rightarrow \bigoplus_{i=1}^h V^{\Gamma_i},~f\mapsto (f(\alpha_1),...,f(\alpha_h))
\end{equation}
is an isomorphism of vector spaces, where $V^{\Gamma_i}$ denotes the $\Gamma_i$ fixed points in $V$. In particular $M(V,K)$ is finite-dimensional.
\end{enumerate}
\end{Proposition}
\begin{Remark}
 The groups $\Gamma_i$ are discrete subgroups of the compact group $\MU_n(\MR)$ hence finite. When $K=K_L, \Gamma_i=\Aut(\alpha_i L)$. For $V=\MC$ the trivial representation, and in the notation of the preceding proposition, we have $V^{\Gamma_i}=V$ for all $i$, so there is a natural isomorphism between $M(V,K)$ and the space of $\MC$-valued functions on $\Sigma_K$.
\end{Remark}

 \section{Hecke operators for unitary groups}
\subsection{Hecke operators}
We keep the notation from the previous section.

In addition to being a finite-dimensional $\MC$-vector space, the space of algebraic modular forms $M(V,K)$ also carries the structure of a module over the Hecke algebra of $\MU_n$ with respect to $K$. The Hecke algebra $H_K=H(\MU_n,K)$ is the $\MC$-algebra of all locally constant, compactly supported functions $\MU_n(\hQ) \rightarrow \MC$ which are $K$-bi-invariant. The multiplication in $H_K$ is given by convolution with respect to the (unique) Haar measure giving the compact group $K$ measure $1$.
The algebra $H_K$ has a canonical basis given by the characteristic functions of the double cosets with respect to $K$, $\ind_{K\gamma K}$, $\gamma \in \MU_n(\hQ)$.

 The action of $\ind_{K\gamma K}\in H_K$ on $M(V,K)$ is given as follows: Decompose $K\gamma K= \sqcup_{i} \gamma_i K$ into disjoint $K$-left cosets, then $\ind_{K \gamma K}$ acts via the operator $T(\gamma)=T(K\gamma K) \in \End(M(V,K))$ defined by
\begin{equation}
(T(\gamma)f)(x)=\sum_i f(x\gamma_i) \text{ for } f \in M(V,K),x\in \MU_n(\hQ).
\end{equation}

We can extend $T$ linearly to $H_K$ to obtain a homomorphism of $\MC$-algebras. Moreover,
\begin{Proposition}[\protect{\cite[Prop. (6.11)]{GrossAlgebraicModularForms}}]\label{AdjointOperator}
 $M(V,K)$ is a semisimple $H_K$-module.
\end{Proposition}

\subsection{Neighbours and Hecke operators}
We now want to describe how one can compute certain Hecke operators.
\begin{Definition}
 Let $L \subset V_n$ be an integral Hermitian lattice and let $\fp \subset \CO_E$ be a nonzero prime ideal of $\CO_E$ such that $\fp \nmid \partial(L)$. An integral Hermitian lattice $L'\subset V_n$ is called a $\fp$-neighbour of $L$ if $L/(L \cap L') \cong \CO_E/\fp$ and $L'/(L \cap L') \cong \CO_E/\overline{\fp}$. We denote the (finite) set of $\fp$-neighbours of $L$ by $N(L,\fp)$.
\end{Definition}

The ``neighbour method'' (for obtaining the classes in a genus) was introduced by Kneser \cite{Kneser} for integral Euclidean lattices, and extended to the Hermitian case by Iyanaga and Hoffmann in special cases \cite{Iyanaga},\cite{Hoffmann} and by Schiemann in general \cite{Schiemann}. See the beginning of \S5 of \cite{GreenbergVoightLatticeMethods} for further comments on the development of the method.
\begin{Proposition}
\begin{enumerate}
\item Suppose that $n\geq 3$, and take $L$, $\fp$ as above. Then $N(L,\fp)\subseteq\genus(L)$.
\item Take an element $\gamma \in \MU_n(\hQ)$ such that $\gamma L \in N(L,\fp)$. Then $N(L,\fp) = (K_L\gamma K_L) \{L\}$.
\end{enumerate}
\end{Proposition}
\begin{proof} The first part is an immediate consequence of \cite[Lemma 2.8]{Schiemann}. If we take $\kappa\in K_L$ then clearly $L'\in N(L,\fp)\implies\kappa L'\in N(\kappa L,\fp)=N(L,\fp)$, hence $(K_L\gamma K_L) \{L\}\subseteq N(L,\fp)$. Conversely, since $N(L,\fp)\subseteq\genus(L)$, if $L'\in N(L,\fp)$ then $L'=gL$ for some $g\in\MU_n(\hQ)$. Clearly only the component at $p$ of $\gamma$ matters, and we may take the components at all other primes to be the identity. Likewise for $g$. By the Cartan decomposition, $\MU_n(\MQ_p)=\sqcup K_L\lambda(p)K_L$, where $\lambda$ runs over non-negative co-characters of a maximal split torus of $\MU_n(\MQ_p)$ \cite[Prop. 4.2]{GreenbergVoightLatticeMethods}. Consideration of invariant factors of $L'$ and $L$ shows that $g$ can only be in $K_L\gamma K_L$.
\end{proof}
We call $T(K_L \gamma K_L)$ the neighbouring operator at $\fp$, and denote it by $T_\fp$.

Propositions 5.3, 5.4 (and 5.6) of \cite{GreenbergVoightLatticeMethods} describe an algorithm for computing all $\fp$-neighbours of a given lattice $L$. In addition we can test Hermitian lattices for isometry by employing the Plesken-Souvignier algorithm \cite{PleskenSouvignier}. (These two steps are essentially Sections 4.1 and 4.2 of \cite{Schiemann}.) Knowing this, we can compute the Hecke operator corresponding to the $\fp$-neighbours of a given lattice as follows.
\begin{Proposition}\label{NeighbourOperator}
 Let $L$ be an integral Hermitian lattice with adelic stabilizer $K_L < \MU_n(\hQ)$ and $\fp \subset \CO_E$ a nonzero prime ideal such that $\fp \nmid \partial(L)$. Choose a system $\alpha_i, 1\leq i \leq h$ for $\mathrm{U}_n \backslash \MU_n(\hQ) / K_L$, so $L_i:=\alpha_i L, 1 \leq i \leq h,$ forms a system of representatives for the classes in $\genus(L)$. With respect to the natural basis of $M(\mathrm{triv},K_L)$ (see below), the operator $T_{\fp}$ has the matrix representation $(t_{i,j})_{i,j=1}^h$ with
\begin{equation}
t_{i,j} =  |\{ L' \in N(L_i,\fp) ~|~ L' \cong L_j \}|.
\end{equation}
\end{Proposition}
\begin{proof}
The natural basis for $M(\mathrm{triv},K_L)$ consists of $\mathrm{U}_n$-$K_L$ invariant maps $F_i:\MU_n(\hQ) \rightarrow \MC$ with $F_i(\alpha_j) = \delta_{i,j}$. We decompose
\begin{equation}
 K_L \gamma K_L = \bigsqcup_{j=1}^r \gamma_j K_L,
\end{equation}
so
\begin{equation}
 N(L,\fp) = \{ \gamma_j L ~|~ 1 \leq j \leq r \}.
\end{equation}
Then the $(i,j)$-entry of $T(K_L \gamma K_L)$ with respect to the given basis is the coefficient of $F_i$ in $T(K_L \gamma K_L)F_j$, which is
\begin{equation}
 \begin{split}
  T(K_L \gamma K_L)F_j(\alpha_i) &= \sum_{a=1}^r F_j(\alpha_i \gamma_a)\\
 &= |\{ 1\leq a \leq r ~|~ \alpha_i \gamma_a \in \mathrm{U}_n  \alpha_j  K_L \}| \\
 &= |\{ 1\leq a \leq r ~|~ \alpha_i \gamma_a L \cong \alpha_j L \}|\\
 &= |\{ L' \in N(L,\fp)~|~ \alpha_i L' \cong L_j \}| \\
 &=|\{ L' \in N(L_i,\fp)~|~  L' \cong L_j \}|.
 \end{split}
\end{equation}
This proves the assertion.
\end{proof}

While this algorithm works perfectly well in the cases we are primarily interested in, we also want to describe an alternative method, which often takes significantly less time. This alternative method has the added benefit of computing a system of representatives of a second genus of lattices along the way. The following proposition makes explicit in our setting the method implicit in the general \cite[Corollary 4.5]{SchoennenbeckSimultaneous}.

\begin{Proposition}\label{IntertwiningOperator}
 Let $L$ and $\fp$ be as in Proposition \ref{NeighbourOperator}, where in addition we assume that $\fp$ divides a prime of $\MQ$ that is either inert or ramified in $E$. Let $L' =L \cap N'$ for some $N' \in N(L,\fp)\}$. Choose representatives $L_i, 1\leq i \leq h$ and $L'_j, 1 \leq j \leq h',$ for $\genus(L)$ and $\genus(L')$, respectively. Let $S_\fp=(s_{i,j}) \in \MZ^{h \times h'}$ be the matrix with entries
\begin{equation}
 s_{i,j} = |\{X \subset L_i~|~ X \cong L'_j\}|.
\end{equation}
In addition let $d:=|\{ L \cap N'~|~N' \in N(L,\fp)\}|$.
Let $S'_\fp$ be the matrix
\begin{equation}
 S'_\fp:=\diag(|\Aut(L'_1)|,...,|\Aut(L'_{h'})|) \cdot S^{t} \cdot \diag(|\Aut(L_1)|^{-1},...,|\Aut(L_{h})|^{-1}).
\end{equation}

Then the operator $T_{\fp}$ is represented (with respect to the natural basis $\{F_i\}$ as above) by the matrix
\begin{equation}
 S_\fp\cdot S'_\fp  - d \cdot I_{h}.
\end{equation}
\end{Proposition}
\begin{proof}
 The proof works analogously to that of Proposition \ref{NeighbourOperator}. One needs to see that the $(j,i)$-entry of $S'_\fp$
 counts the lattices above $L'_j$ which are isomorphic to $L_i$. This however is a simple counting argument (cf. \cite{NebeBachoc} and \cite[Lemma 4.2]{SchoennenbeckSimultaneous}).
 Then multiplying $S_\fp$ by $S'_\fp$ counts the number of ways of getting from each $L_i$ to a neighbour isomorphic to $L_j$ via common sublattices in each class of $\genus(L')$ (on the other side of a bipartite graph), except we have to subtract off the number of ways of getting from $L_i$ back to itself rather than to a neighbour isomorphic to itself.
\end{proof}

\begin{Remark}
 If one wants to employ Proposition \ref{IntertwiningOperator} in practice it is not necessary to have a system of representatives of $\genus(L')$ already at hand. Instead such a system can be found along the way by computing the relevant sublattices of the representatives for $\genus(L)$.
\end{Remark}

\subsection{Computational results}\label{compres}
In this subsection we present the results of our computations of Hecke operators for certain genera of Hermitian lattices for $E=\MQ(\sqrt{-3})$.

We start by computing the neighbouring operator $T_{(2)}$ acting on the space $M(\mathrm{triv},K_L)$ where the genus of $L$ comprises all $\langle \sqrt{-3}\rangle$-modular lattices in $V_{12}$ and was classified in \cite{HentschelKriegNebeClassification}, where they are called Eisenstein lattices. We shall meet this genus again briefly in \S \ref{Eisenstein} below, and at the end of \S 7. It decomposes into $5$ isometry classes which we take in the order of \cite[Thm. 2]{HentschelKriegNebeClassification}. In particular, the cardinalities of the automorphism groups are (in this order)
\begin{equation}
 22568879259648000, 8463329722368, 206391214080, 101016305280,\text{ and } 2690072985600.
\end{equation}
 Following Proposition \ref{NeighbourOperator} one computes the basis representation of $T_{(2)}$ as
\begin{equation}
 \begin{pmatrix} 65520  &  3888000  &  1640250  &  0  &  0  \\
 1458  &  516285  &  3956283  &  1119744  &  0  \\
 15  &  96480  &  2467899  &  2998272  &  31104  \\
 0  &  13365  &  1467477  &  3935781  &  177147  \\
 0  &  0  &  405405  &  4717440  &  470925  \\
\end{pmatrix}.
\end{equation}

Alternatively we employ Proposition \ref{IntertwiningOperator}. The second genus which we compute along the way consists of lattices which are of index $4$ (elementary divisors $2\cdot \CO_E$) in lattices of the given genus. It decomposes into $25$ isometry classes  with corresponding automorphism group cardinalities (in the order in which we found the representatives)
\begin{equation}
\begin{split}
 &501530650214400, 4701849845760, 3715041853440, 11609505792, 27518828544, \\
&705277476864, 9795520512, 181398528, 967458816, 15116544,\\
& 4478976, 95551488, 103195607040, 859963392, 23887872,\\
 &20155392, 839808, 186624, 13271040, 524880,\\
&1530550080, 9447840, 233280, 1140480, \text{ and } 246343680.
\end{split}
\end{equation}

After computation we find
\begin{equation}
 S'_2=\begin{pmatrix} 3  &  0  &  0  &  0  &  0  \\
 3  &  0  &  0  &  0  &  0  \\
 3  &  0  &  0  &  0  &  0  \\
 1  &  2  &  0  &  0  &  0  \\
 1  &  0  &  2  &  0  &  0  \\
 0  &  3  &  0  &  0  &  0  \\
 0  &  3  &  0  &  0  &  0  \\
 0  &  3  &  0  &  0  &  0  \\
 0  &  3  &  0  &  0  &  0  \\
 0  &  1  &  0  &  2  &  0  \\
 0  &  1  &  2  &  0  &  0  \\
 0  &  2  &  1  &  0  &  0  \\
 0  &  0  &  3  &  0  &  0  \\
 0  &  0  &  3  &  0  &  0  \\
 0  &  0  &  3  &  0  &  0  \\
 0  &  0  &  3  &  0  &  0  \\
 0  &  0  &  3  &  0  &  0  \\
 0  &  0  &  1  &  2  &  0  \\
 0  &  0  &  1  &  0  &  2  \\
 0  &  0  &  2  &  1  &  0  \\
 0  &  0  &  0  &  3  &  0  \\
 0  &  0  &  0  &  3  &  0  \\
 0  &  0  &  0  &  3  &  0  \\
 0  &  0  &  0  &  2  &  1  \\
 0  &  0  &  0  &  0  &  3  \\
\end{pmatrix},
\end{equation}
where we only write down $S'_2$ for the sake of readability. In particular the number $d$ from Proposition \ref{IntertwiningOperator} is $2796885$ (as it is the sum of the entries of any of the rows of $S_2$) and we obtain the same representation for $T_{(2)}$ as before.

We now present the Hecke operators we are primarily interested in. To that end let $L \subset V_{12}$ be a unimodular Hermitian lattice of rank $12$. The genus of $L$ consists of $20$ isometry classes which were classified in \cite{FeitSomeLattices}. We  consider them in the following order: The first $11$ are the indecomposable unimodular lattices in the same order as \cite[Table II]{FeitSomeLattices}, then the $12$-dimensional standard lattice, then the $7$ direct sums of lattices in \cite[Table I]{FeitSomeLattices} with standard lattices of appropriate rank (again in the same order as in the source table), and finally the direct sum of two copies of the lattice called $U_6$ in \cite{FeitSomeLattices}.

Employing Proposition \ref{IntertwiningOperator} we managed to compute the Hecke operators $T_{(2)}$ and $T_{(\sqrt{-3})}$ acting on $M(\mathrm{triv},K_L)$ which are given by the following matrices. This time the second genus contains $344$ classes (when $\fp=(2)$) or $96$ classes (when $\fp=(\sqrt{-3})$). The computation of $T_{(\sqrt{-3})}$ took about $5$ minutes, while that of $T_{(2)}$ took about an hour, using Proposition \ref{IntertwiningOperator} and no more than $8$ GB of memory, on a standard desktop computer. We also obtained the same matrices using Proposition \ref{NeighbourOperator}. This took much longer (though less than a day), and when we tried to repeat the computations recently, they broke off due to memory constraints. The code for our algorithms is publicly available on \url{https://github.com/schoennenbeck/hecke_operators}.
\begin{footnotesize}

\begin{equation}
\rotatebox{90}{
$T_2:
\left(\begin{smallmatrix} 3389169& 0& 49& 18088& 18088& 588& 661920& 278565& 4116& 215600& 918750& 0& 0& 0& 70& 1302& 441& 45570& 41454& 0\\
0& 72810& 118098& 1358127& 1358127& 1082565& 1082565& 0& 449064& 0& 0& 0& 0& 0& 13365& 0& 0& 59049& 0& 0\\
473088& 2048& 129031& 575488& 575488& 110880& 1013760& 2113188& 473088& 0& 46200& 1& 0& 990& 0& 12012& 0& 67584& 0& 924\\
2217072& 299& 7306& 262990& 154297& 37323& 1602315& 1003002& 72072& 0& 150150& 0& 0& 0& 715& 4290& 0& 45903& 36036& 0\\
2217072& 299& 7306& 154297& 262990& 37323& 1602315& 1003002& 72072& 0& 150150& 0& 0& 0& 715& 4290& 0& 45903& 36036& 0\\
870912& 2880& 17010& 451008& 451008& 174303& 2099520& 1197504& 161280& 0& 85050& 0& 63& 135& 0& 8505& 1890& 0& 72576& 126\\
3063744& 9& 486& 60507& 60507& 6561& 1104750& 523908& 19656& 120960& 544320& 0& 0& 0& 153& 1944& 756& 38637& 46872& 0\\
2910720& 0& 2287& 85504& 85504& 8448& 1182720& 720060& 34304& 102400& 373800& 0& 2& 60& 0& 2580& 900& 49152& 35328& 1\\
1653372& 324& 19683& 236196& 236196& 43740& 1705860& 1318761& 196713& 97200& 0& 0& 0& 0& 810& 0& 6075& 39366& 39366& 108\\
3464208& 0& 0& 0& 0& 0& 419904& 157464& 3888& 337437& 1121931& 0& 0& 0& 0& 0& 1458& 52488& 34992& 0\\
3456000& 0& 18& 4608& 4608& 216& 442368& 134568& 0& 262656& 1199763& 0& 0& 9& 0& 1296& 108& 41472& 46080& 0\\
0& 0& 354294& 0& 0& 0& 0& 0& 0& 0& 0& 27126& 449064& 2165130& 0& 2598156& 0& 0& 0& 0\\
0& 0& 0& 0& 0& 183708& 0& 826686& 0& 0& 0& 567& 59931& 153090& 725760& 1377810& 306180& 0& 1959552& 486\\
0& 0& 21870& 0& 0& 19440& 0& 1224720& 0& 0& 510300& 135& 7560& 81045& 207360& 831060& 340200& 1866240& 483840& 0\\
1417176& 486& 0& 118098& 118098& 0& 669222& 0& 40824& 0& 0& 0& 3402& 19683& 123111& 551124& 183708& 1062882& 1285956& 0\\
1714176& 0& 1638& 46080& 46080& 7560& 552960& 325080& 0& 0& 453600& 1& 420& 5130& 35840& 275229& 77112& 1036800& 1016064& 0\\
1679616& 0& 0& 0& 0& 4860& 622080& 328050& 57600& 345600& 109350& 0& 270& 6075& 34560& 223074& 129753& 1119744& 933120& 18\\
2577960& 6& 396& 21186& 21186& 0& 472230& 266112& 5544& 184800& 623700& 0& 0& 495& 2970& 44550& 16632& 721215& 634788& 0\\
2466936& 0& 0& 17496& 17496& 2916& 602640& 201204& 5832& 129600& 729000& 0& 27& 135& 3780& 45927& 14580& 667764& 688437& 0\\
0& 0& 964467& 0& 0& 857304& 0& 964467& 2709504& 0& 0& 0& 1134& 0& 0& 0& 47628& 0& 0& 49266
\end{smallmatrix} \right)$}
\end{equation}

\begin{equation}
\rotatebox{90}{
$T_{\sqrt{-3}}: \left(\begin{smallmatrix} 163898& 0& 0& 693& 693& 14& 31010& 12348& 98& 9800& 44100& 0& 0& 0& 0& 70& 0& 1960& 1764& 0\\
0& 1322& 59049& 59049& 59049& 0& 0& 0& 85932& 0& 0& 1& 0& 0& 660& 0& 0& 0& 0& 1386\\
0& 1024& 3808& 57344& 57344& 29568& 56320& 59136& 0& 0& 0& 0& 0& 0& 0& 880& 0& 1024& 0& 0\\
84942& 13& 728& 6292& 19604& 1716& 77935& 66066& 6006& 0& 0& 0& 0& 0& 0& 286& 0& 2860& 0& 0\\
84942& 13& 728& 19604& 6292& 1716& 77935& 66066& 6006& 0& 0& 0& 0& 0& 0& 286& 0& 2860& 0& 0\\
20736& 0& 4536& 20736& 20736& 5288& 53760& 111132& 26880& 0& 0& 0& 0& 40& 0& 0& 336& 0& 2240& 28\\
143532& 0& 27& 2943& 2943& 168& 57422& 30618& 1456& 6720& 17010& 0& 0& 1& 12& 0& 84& 1944& 1568& 0\\
129024& 0& 64& 5632& 5632& 784& 69120& 37280& 1024& 0& 14400& 0& 0& 0& 0& 160& 0& 1024& 2304& 0\\
39366& 62& 0& 19683& 19683& 7290& 126360& 39366& 9798& 1800& 0& 0& 2& 0& 120& 0& 0& 0& 2916& 2\\
157464& 0& 0& 0& 0& 0& 23328& 0& 72& 11696& 69984& 0& 0& 0& 16& 0& 0& 0& 3888& 0\\
165888& 0& 0& 0& 0& 0& 13824& 5184& 0& 16384& 61280& 0& 0& 0& 0& 0& 48& 2304& 1536& 0\\
0& 6144& 0& 0& 0& 0& 0& 0& 0& 0& 0& 2640& 88704& 0& 168960& 0& 0& 0& 0& 0\\
0& 0& 0& 0& 0& 0& 0& 0& 21504& 0& 0& 112& 5000& 22680& 53760& 0& 90720& 0& 72576& 96\\
0& 0& 0& 0& 0& 5760& 46080& 0& 0& 0& 0& 0& 1120& 5168& 23040& 77760& 0& 0& 107520& 0\\
0& 24& 0& 0& 0& 0& 52488& 0& 6048& 20160& 0& 1& 252& 2187& 9356& 39366& 17010& 78732& 40824& 0\\
92160& 0& 120& 3072& 3072& 0& 0& 20160& 0& 0& 0& 0& 0& 480& 2560& 20600& 6720& 69120& 48384& 0\\
0& 0& 0& 0& 0& 864& 69120& 0& 0& 0& 48600& 0& 80& 0& 3200& 19440& 7640& 31104& 86400& 0\\
110880& 0& 6& 1320& 1320& 0& 23760& 5544& 0& 0& 34650& 0& 0& 0& 220& 2970& 462& 42812& 42504& 0\\
104976& 0& 0& 0& 0& 90& 20160& 13122& 432& 14400& 24300& 0& 1& 30& 120& 2187& 1350& 44712& 40568& 0\\
0& 25088& 0& 0& 0& 190512& 0& 0& 50176& 0& 0& 0& 224& 0& 0& 0& 0& 0& 0& 448
\end{smallmatrix} \right)$}
\end{equation}
\end{footnotesize}
\section{Eigenvectors and automorphic representations}
\subsection{The eigenvectors}\label{eigenvectors}
We consider still the genus of $20$ classes of unimodular Hermitian lattices of rank $12$.
Recall from Proposition \ref{AdjointOperator} that $M(\mathrm{triv},K_L)$ is a semisimple $H_K$-module. In a given simple submodule, all vectors share the same $T_{\fp}$ eigenvalue, for any given neighbouring operator $T_{\fp}$. We found a basis $\{v_1,v_2\ldots,v_{20}\}$ of $M(\mathrm{triv},K_L)$, simultaneous eigenvectors for $T_{(2)}$ and $T_{(\sqrt{-3})}$. We write $\lambda_i(T(\gamma))$ for the eigenvalue of $T(\gamma)$ acting on $v_i$. We order them as in the following table, which presents the eigenvalues for $T_{(2)}$ and $T_{(\sqrt{-3})}$. There is only one eigenvalue for $T_{(2)}$ whose eigenspace is not $1$-dimensional. In fact $T_{(2)}$ and $T_{(\sqrt{-3})}$ have a common $2$-dimensional eigenspace, spanned by $v_{19}$ and $v_{20}$, though looking at the last column of the table (to be explained later) we would expect it to be broken up by $T_{(2+\sqrt{-3})}$, since $\psi$ and $\overline{\psi}$ differ on the factors of the split prime $7$. So we know that at least $18$ of the simple $H_K$-submodules of $M(\mathrm{triv},K_L)$ are $1$-dimensional, and suspect that they all are. The eigenvalues and eigenvectors were computed using the $20$-by-$20$ matrices referred to above.
\vskip15pt
 \begin{tabular}{|c|c|c|c|}\hline$\mathbf{i}$ & $\lambda_i\left(T_{(2)}\right)$ & $\lambda_i\left(T_{(\sqrt{-3})}\right)$ & Global Arthur parameters (conjectural)\\\hline
 $\mathbf{1}$ & $5593770$ & $266448$ & $[12]$\\$\mathbf{2}$ & $1395945$ & $89552$ & $\Delta_{11}\oplus [10]$\\$\mathbf{3}$ & $1401453$ & $88328$ & ${}^3\Delta_{11}\oplus [10]$\\$\mathbf{4}$ & $357525$ & $30032$ & ${}^3\Delta_{10}[2]\oplus [8]$\\$\mathbf{5}$ & $348453$ & $29528$ & $\Delta_{11}\oplus {}^3\Delta_9\oplus [8]$\\$\mathbf{6}$ & $91845$ & $9368$ & ${}^3\Delta_{10}[2]\oplus {}^3\Delta_7\oplus [6]$\\$\mathbf{7}$ & $90873$ & $10664$ & ${}^3\Delta_{11}\oplus {}^3\Delta_8[2]\oplus [6]$\\$\mathbf{8}$ & $85365$ & $11888$ & $\Delta_{11}\oplus{}^3\Delta_8[2]\oplus [6]$\\$\mathbf{9}$ & $23805$ & $7568$ & ${}^3\Delta_8[4]\oplus [4]$\\$\mathbf{10}$ & $40005$ & $1808$ & ${}^3\Delta_{10}[2]\oplus {}^3\Delta_6[2]\oplus [4]$\\$\mathbf{11}$ & $30933$ & $1304$ & $\Delta_{11}\oplus {}^3\Delta_9\oplus{}^3\Delta_6[2]\oplus [4]$ \\$\mathbf{12}$ & $23319+162\sqrt{193}$ & $4148+ 36\sqrt{193}$ & $\Delta_{11,5}^{(2)}\oplus{}^3\Delta_8[2]\oplus [4]$\\$\mathbf{13}$ & $23319-162\sqrt{193}$ & $4148- 36\sqrt{193}$ & " \\$\mathbf{14}$ & $46485$ & $-4528$ & $\Delta_{11}\oplus {}^3\Delta_6[4]\oplus [2]$\\$\mathbf{15}$ & $51993$ & $-5752$ & ${}^3\Delta_{11}\oplus {}^3\Delta_6[4]\oplus [2]$\\$\mathbf{16}$ & $11925$ & $-1072$ & $\Delta_{11}\oplus \Delta_{9,3}\oplus{}^3\Delta_6[2]\oplus [2]$\\$\mathbf{17}$ & $176085$ & $-18928$ & ${}^3\Delta_6[6]$\\$\mathbf{18}$ & $-5355$ & $728$ & $\Delta_{11}\oplus \Delta_{9,1}\oplus {}^3\Delta_5[3]$\\$\mathbf{19}$ & $108693$ & $-13312$ & $({}^3\Delta_5\otimes\psi_6)\oplus \psi_6[4]\oplus\overline{\psi}_6[6]$\\$\mathbf{20}$ & $108693$ & $-13312$ & $({}^3\Delta_5\otimes\overline{\psi}_6)\oplus \overline{\psi}_6[4]\oplus\psi_6[6]$\\\hline
 \end{tabular}
 \vskip10pt
 In the remainder of this section we attempt to explain the last column of the table.
 \subsection{Automorphic representations of $\MU_{12}(\MA_{\MQ})$ and their local Langlands parameters}
 Each $v_i$ may be thought of as a complex-valued function on $\MU_{12}(\MQ)\backslash \MU_{12}(\MA_{\MQ})$, right-invariant under $K=K_{\infty}K_L$, where $K_{\infty}:=\MU_{12}(\MR)$. Under the right-translation action of $\MU_{12}(\MA_{\MQ})$, it generates an infinite-dimensional automorphic representation $\pi_i$ of $\MU_{12}(\MA_{\MQ})$. There is a one-to-one correspondence between simple $H_K$-submodules of $M(\mathrm{triv},K_L)$ and automorphic representations of $\MU_{12}(\MA_{\MQ})$ such that $\pi_{\infty}$ is trivial and with a $K_L$-fixed vector \cite[Proposition 2.5]{GreenbergVoightLatticeMethods}. It follows that the $\pi_i$ are all distinct, with the possible (but unlikely) exception of $\pi_{19}$ and $\pi_{20}$.

 For each local Weil group $W_{\MR}$ and $W_{\MQ_p}$ of $\MQ$ there is associated to $\pi_i$ a Langlands parameter, a homomorphism from that group to the local $L$-group $\GL_{12}(\MC)\rtimes W_{\MR}$ or $\GL_{12}(\MC)\rtimes W_{\MQ_p}$ of $\MU_{12}$. Restricting to the local Weil group $W_{\MC}$ or $W_{E_{\FP}}$ of $E$, and projecting to $\GL_{12}(\MC)$, we obtain Langlands parameters
 $$c_{\infty}(\tilde{\pi_i}):W_{\MC}\rightarrow \GL_{12}(\MC),\,\,\,c_{E_{\FP}}(\tilde{\pi_i}):W_{E_{\FP}}\rightarrow\GL_{12}(\MC),$$
 defined up to conjugation in $\GL_{12}(\MC)$, which is here playing the role of the Langlands dual of $\GL_{12,E}$.
 See \cite{Mok} (following (2.2.3)) for this ``standard base-change of $L$-parameters''.
 Now $W_{\MC}=\MC^{\times}$, and it is a consequence of the fact that $v_i$ is scalar-valued that (up to conjugation)
 $$c_{\infty}(\tilde{\pi_i}):z\mapsto \diag\left((z/\overline{z})^{11/2},(z/\overline{z})^{9/2},\ldots,(z/\overline{z})^{-11/2}\right).$$
 At a prime $\FP$ dividing $p$ for which both $\MU_{12}$ and $\pi_i$ are unramified (i.e. $p\neq 3$, given our choice of $E$ and $K_L$), $c_{E_{\FP}}(\tilde{\pi_i})$ is determined by $\mathrm{Frob}_{\FP}\mapsto t_{\FP}(\tilde{\pi}_i)$ (the Satake parameter at $\FP$).
 This determines $\lambda_i(T_{\FP})$, by the formulas
 $$\lambda_i(T_{\FP})=\begin{cases} (N\FP)^{11/2}\tr(t_{\FP}(\tilde{\pi}_i))+\frac{p^{12}-1}{p+1} & \text{for $(p)=\FP$ inert};\\
 (N\FP)^{11/2}\tr(t_{\FP}(\tilde{\pi}_i)) & \text{for $(p)=\FP\overline{\FP}$ split}.\end{cases}$$
 In the split case, where $\MU_{12}(\MQ_p)\simeq\GL_{12}(\MQ_p)$, this is a direct consequence of a formula of Tamagawa \cite{Tamagawa},\cite[$i=1$ in (3.14)]{GrossSatake}. In the inert case, where $\MU_{12}(\MQ_p)\simeq \MU(6,6)(\MQ_p)$, it may be justified assuming a coset decomposition like that for $\MU(2,2)$ in \cite[(5.7)]{Klosin}, combined with \cite[IV. (33),(35),(39)]{Cartier}.
 \subsection{Global Arthur parameters}
 A complete description of those automorphic representations, of a quasi-split unitary group $\GG^*$, occurring discretely in $L^2(\GG^*(\MQ)\backslash \GG^*(\MA_{\MQ}))$, was given by Mok \cite[Theorem 2.5.2]{Mok}. This was extended to general unitary groups
 (including $\MU_{12}$) by Kaletha, Minguez, Shin and White \cite[Theorem* 1.7.1]{KMSW}, conditional on what will be written up in later papers of Kaletha, Minguez and Shin, and of Chaudouard and Laumon. (See the discussion on \cite[p.6]{KMSW}.) Part of this description is that to such an automorphic representation is attached a ``global Arthur parameter'', a formal unordered sum of the form $\oplus_{k=1}^m\Pi_k[d_k]$, where $\Pi_k$ is a cuspidal automorphic representation of $\GL_{n_i}(\MA_{E})$, $d_k\geq 1$ and $\sum_{k=1}^mn_kd_k=N=12$. Before explaining the guesses in the final column of the table, we fix some notation.

 {\em Level $1$. }Let $f$ be a cuspidal Hecke eigenform of weight $k$ for $\SL_2(\MZ)$. There is an associated cuspidal automorphic representation $\Pi_f$ of $\GL_2(\MA_{\MQ})$, with base-change $\tilde{\Pi_f}$ of $\GL_2(\MA_{E})$. We have
 $$c_{\infty}(\tilde{\Pi_f}):z\mapsto \diag\left((z/\overline{z})^{(k-1)/2},(z/\overline{z})^{(1-k)/2}\right),$$
 and
 $$t_{\FP}(\tilde{\Pi_f})=\begin{cases} \diag(\alpha,\alpha^{-1}) & \text{$p$ split};\\ \diag(\alpha^2,\alpha^{-2}) & \text{$p$ inert},\end{cases}$$
 where $a_p(f)=p^{(k-1)/2}(\alpha+\alpha^{-1})$ and $|\alpha|=1$.
 In the table, $\tilde{\Pi_f}$ is denoted $\Delta_{k-1}$, the subscript coming from the exponents in $c_{\infty}(\tilde{\Pi_f})$.
 For example when $k=12$ and $f=\Delta=\sum_{n=1}^{\infty}\tau(n)q^n$, we have $\Delta_{11}$.

 {\em Level $3$. } Similarly for a newform $f\in S_k(\Gamma_0(3))$ we denote $\tilde{\Pi_f}$ by ${}^3\Delta_{k-1}$, e.g. ${}^3\Delta_{11},{}^3\Delta_9,{}^3\Delta_7$ and ${}^3\Delta_5$. Although $S_{10}(\Gamma_0(3))$ is $2$-dimensional, we reserve ${}^3\Delta_9$ for the base change associated to just one of the normalised eigenforms, it being the only one that appears to actually occur in the global Arthur parameters of any of our $\pi_i$.

 For a Hecke eigenform $f\in S_k(\Gamma_0(3),\chi_{-3})$ (where $\chi_{-3}$ is the quadratic character attached to $E$), with $k$ odd, we have ${}^3\Delta_{k-1}$ , e.g ${}^3\Delta_{10}$ and ${}^3\Delta_{8}$. Note that each of $S_{11}(\Gamma_0(3),\chi_{-3})$ and $S_{9}(\Gamma_0(3),\chi_{-3})$ is spanned by a conjugate pair of Hecke eigenforms, sharing the same base-change. Note also that $S_7(\Gamma_0(3),\chi_{-3})$ is spanned by a Hecke eigenform $f$ of CM type. The base change $\tilde{\Pi_f}$ is, in this case, not cuspidal, but we still use ${}^3\Delta_6$ as a shorthand for $\psi_6\oplus\overline{\psi}_6$, where $\psi_6$ is an everywhere-unramified, cuspidal, automorphic representation of $\GL_1(\MA_E)$, given by $\psi_6(z)=z^{-6}$ for $z\in \MC^{\times}$ (embedded in $\MA_E^{\times}$ by putting $1$ in all the other components) and $\psi_6(\pi_{\FP})=\alpha_{\FP}^6$, where $\pi_{\FP}\in E_{\FP}^{\times}$ is a uniformiser at $\FP$ and $(\alpha_{\FP})=\FP$. Since the group of units in $\CO_E$ has order $6$, this is well-defined, independent of the choice of $\alpha_{\FP}$.

 {\em Square brackets. }For $d\geq 1$, $[d]$ is an automorphic representation of $\GL_d(\MA_E)$, occurring discretely in $L^2(\GL_d(E)\backslash \GL_d(\MA_E))$, with
 $$c_{\infty}([d]):z\mapsto\mathrm{Sym}^{d-1}\left(\diag \left((z/\overline{z})^{1/2},(z/\overline{z})^{-1/2}\right)\right)$$ and
 $$c_{E_{\FP}}([d]):\mathrm{Frob}_{\FP}\mapsto \mathrm{Sym}^{d-1}\left(\diag \left((N\FP)^{1/2},(N\FP)^{-1/2}\right)\right).$$
 Given a cuspidal automorphic representation $\Pi$ of $\GL_{n}(\MA_E)$, there is a discrete automorphic representation $\Pi[d]$ of $\GL_{nd}(\MA_E)$, whose Langlands parameters are tensor products of those of $\Pi$ and of $[d]$.

 Now we may associate local Langlands parameters to each global Arthur parameter $\oplus_{k=1}^m\Pi_k[d_k]$ by taking a direct sum of those just attached to each constituent $\Pi_k[d_k]$. To list such a global Arthur parameter in the $i^{\mathrm{th}}$ row of the table is to conjecture that its associated local Langlands parameters are the $c_{\infty}(\tilde{\pi_i})$ and $c_{E_{\FP}}(\tilde{\pi_i})$ introduced in the previous subsection. If $A_i$ is the conjectured global Arthur parameter in the $i^{\mathrm{th}}$ row of the table, we denote by $t_{\FP}(A_i)$ the conjugacy class obtained by applying the associated $c_{E_{\FP}}$ to $\mathrm{Frob}_{\FP}$ (for $p\neq 3$, so that this is well-defined).
 \subsection{Recovering the Hecke eigenvalues}
 All the entries in the final column of the table must agree with the requirement $$c_{\infty}(\tilde{\pi_i}):z\mapsto \diag\left((z/\overline{z})^{11/2},(z/\overline{z})^{9/2},\ldots,(z/\overline{z})^{-11/2}\right), $$ and indeed they do, as illustrated by the following examples.

 For $\mathbf{1}$, $c_{\infty}([12])=\mathrm{Sym}^{11}\left(\diag \left((z/\overline{z})^{1/2},(z/\overline{z})^{-1/2}\right)\right)$, which is precisely $$\diag\left((z/\overline{z})^{11/2},(z/\overline{z})^{9/2},\ldots,(z/\overline{z})^{-11/2}\right).$$

 For $\mathbf{2}$, $\Delta_{11}$ contributes the exponents $11/2,-11/2$, while $[10]$ contributes the remaining $9/2,7/2,\ldots,-7/2,-9/2$.

 For $\mathbf{7}$, ${}^3\Delta_{11}$ gives $11/2,-11/2$, ${}^3\Delta_8[2]$ gives $9/2,7/2,-7/2,-9/2$, and $[6]$ the remaining $5/2,3/2,1/2,-1/2,-3/2,-5/2$. Note that ${}^3\Delta_8$ would have contributed $8/2,-8/2$, but with the $[2]$ these got ``smeared'' to either side.

 From the conjectural global Arthur parameter $A_i$ of a $\pi_i$ we may compute a putative $\lambda_i(T_{(2)})$, using the formula
 $$\lambda_i(T_{\FP})=(N\FP)^{11/2}\tr(t_{\FP}(\tilde{\pi}_i))+\frac{p^{12}-1}{p+1}$$ with $\FP=(2)$, $N\FP=4$, substituting $t_{\FP}(A_i)$ for $t_{\FP}(\tilde{\pi}_i)$, to see whether the result agrees with the eigenvalue computed using neighbours.
  \begin{Proposition}\label{recoverT2} For all $1\leq i\leq 20$, excluding $i=12,13$, the formula $4^{11/2}\tr(t_{(2)}(A_i))+\frac{2^{12}-1}{3}$, applied to the conjectured global Arthur parameter $A_i$ in the last column of the table, recovers the computed Hecke eigenvalue for $T_{(2)}$ in the second column. For $i=12,13$ it is at least compatible with the computation referred to in Example 3 in \S\ref{algU4}.
 \end{Proposition}
 \begin{proof}

 $\mathbf{1}$. We have $t_{(2)}([12])=\diag(4^{11/2},\ldots,4^{-11/2})$, so get
 $$4^{11/2}\tr(t_{(2)}([12]))+\frac{2^{12}-1}{3}=1+4+4^2+\ldots +4^{11}+\frac{2^{12}-1}{3}=\frac{4^{12}-1}{3} +\frac{2^{12}-1}{3}=5593770,$$
 which does appear in the second column of the first row of the table.

 $\mathbf{2}$. We have $t_{(2)}(\Delta_{11}\oplus [10])=\diag(\alpha^2,4^{9/2},4^{7/2},\ldots,4^{-7/2},4^{-9/2},\alpha^{-2})$, so get
 $$4^{11/2}(\alpha^2+\alpha^{-2})+(4+4^2+\ldots 4^{10})+\frac{2^{12}-1}{3}$$
 $$=((2^{11/2}(\alpha+\alpha^{-1}))^2-2\cdot 2^{11})+4\left(\frac{4^{10}-1}{3}\right)+\frac{2^{12}-1}{3}$$
 $$=(a_2(\Delta)^2-2\cdot 2^{11})+4\left(\frac{4^{10}-1}{3}\right)+\frac{2^{12}-1}{3}$$
 $$=((-24)^2-2\cdot 2^{11})+4\left(\frac{4^{10}-1}{3}\right)+\frac{2^{12}-1}{3}=1395495,$$
using $\Delta=q-24q^2+252q^3\ldots$.

 $\mathbf{7}$. Similarly for ${}^3\Delta_{11}\oplus {}^3\Delta_8[2]\oplus [6]$,
 $$(78^2-2.2^{11})+4((6\sqrt{-14})^2+2\cdot 2^8)(1+4)+4^3\left(\frac{4^6-1}{3}\right)+\frac{2^{12}-1}{3}=90873,$$
 using eigenforms $f=q+78q^2-243q^3\ldots\in S_{12}(\Gamma_0(3))$ and $g=q+6\sqrt{-14}q^2+(45-18\sqrt{-14})q^3\ldots \in S_{9}(\Gamma_0(3),\chi_{-3})$.

 $\mathbf{17}$. Since $\psi_6((2))=\overline{\psi}_6((2))=2^6$, for ${}^3\Delta_6[6]$ we get
 $$(2\cdot 2^6)\left(\frac{4^6-1}{3}\right)+\frac{2^{12}-1}{3}=176085.$$

 $\mathbf{19}$. The Hecke eigenform $q-6q^2+9q^3+\ldots$ spans $S_6(\Gamma_0(3))$, so for $({}^3\Delta_5(3)\otimes\psi_6)\oplus \psi_6[4]\oplus\overline{\psi}_6[6]$ we get
 $$((-6)^2-2\cdot 2^5)(2^6)+4\left(\frac{4^4-1}{3}\right)(2^6)+\left(\frac{4^6-1}{3}\right)(2^6)+\frac{2^{12}-1}{3}=108693.$$

 Similar calculations deal with all the other cases, with $\mathbf{12,13,16,18}$ covered in the examples of the following section.
 \end{proof}
 \subsubsection{Algebraic modular forms for $\MU_4$}\label{algU4}
 With $\GG=\MU_4$ and $V=V_{a,b}$ the representation of highest weight $\lambda=a(e_1-e_4)+b(e_2-e_3)$, where $e_i(\diag(t_1,t_2,t_3,t_4))=t_i$ and $a\geq b\geq 0$, consider the space $M(V,K_L)$, where $L$ is the standard Hermitian lattice $\CO_E^4$, and still $E=\MQ(\sqrt{-3})$. If $v$ is an eigenvector for the Hecke algebra, generating an automorphic representation $\pi$ of $\MU_4(\MA_{\MQ})$, then
 $$c_{\infty}(\tilde{\pi}):z\mapsto\left((z/\overline{z})^{a+3/2},(z/\overline{z})^{b+1/2},(z/\overline{z})^{-b-1/2},(z/\overline{z})^{-a-3/2}\right).$$
 Here $\tilde{\pi}$ is a ``standard base change'' cuspidal automorphic representation of $\GL_4(\MA_{E})$, whose existence would be a consequence of \cite[Theorem* 1.7.1]{KMSW}.
 By \cite[Table 1]{Schiemann}, the class number $|\Sigma_{K_L}|=1$, so $M(V,K_L)=V^{\Gamma}$, where $\Gamma:=\MU_4(\MQ)\cap K_L$. The dimension of $M(V,K_L)$ may be computed using Weyl's character formula.
 \begin{Example} $\mathbf{a=3,b=0}$
 One finds that $\dim(M(V,K_L))=1$. Calculating the trace of $T_{(2)}$ on $M(V,K_L)$ by the method of \cite{DumSimple}, the eigenvalue of $T_{(2)}$ is $1872$. The formula for this is now $(N\FP)^{a+3/2}\tr(t_{\FP}(\tilde{\pi}))+(N\FP)^a\left(\frac{p^4-1}{p+1}\right)$ (previously $a=0$ and $11/2$ was in place of $3/2$), from which we deduce $4^{9/2}\tr(t_{\FP}(\tilde{\pi}))=1872-2^6(2^3-2^2+2-1)=1552$, and then $$4^{11/2}\tr(t_{\FP}(\tilde{\pi}))=4(1872-2^6(2^3-2^2+2-1))=6208.$$ For the cuspidal automorphic representation $\tilde{\pi}$ of $\GL_4(\MA_{E})$ (standard base change from $\MU_4(\MA_{\MQ})$) we write $\Delta_{9,1}$, since
 $$\left(\left(\frac{z}{\overline{z}}\right)^{a+3/2},\left(\frac{z}{\overline{z}}\right)^{b+1/2},\left(\frac{z}{\overline{z}}\right)^{-b-1/2},\left(\frac{z}{\overline{z}}\right)^{-a-3/2}\right)=
 \left(\left(\frac{z}{\overline{z}}\right)^{9/2},\left(\frac{z}{\overline{z}}\right)^{1/2},\left(\frac{z}{\overline{z}}\right)^{-1/2},\left(\frac{z}{\overline{z}}\right)^{-9/2}\right).$$ Now looking back at $\mathbf{18}$, $\Delta_{11}\oplus\Delta_{9,1}\oplus{}^3\Delta_5[3]$ gives us the correct
 $$\lambda_{18}(T_{(2)})=((-24)^2-2\cdot 2^{11})+6208+16((-6)^2-2\cdot 2^5)(1+4+4^2)+\frac{2^{12}-1}{3}=-5355.$$
 \end{Example}
 \begin{Example} $\mathbf{a=3,b=1}$. One finds that $\dim(M(V,K_L))=1$. The eigenvalue of $T_{(2)}$ is $0$. This leads to
$$4^{11/2}\tr(t_{\FP}(\tilde{\pi}))=4(0-2^6(2^3-2^2+2-1))=-1280.$$ For the cuspidal automorphic representation $\tilde{\pi}$ of $\GL_4(\MA_{E})$ we write $\Delta_{9,3}$. Looking back at $\mathbf{16}$, $\Delta_{11}\oplus\Delta_{9,3}\oplus{}^3\Delta_6[2]\oplus [2]$ gives us the correct
 $$\lambda_{16}(T_{(2)})=((-24)^2-2\cdot 2^{11})-1280+16(2\cdot 2^6)(1+4)+2^{10}\left(\frac{4^2-1}{3}\right)+\frac{2^{12}-1}{3}=11925.$$
 Remarkably, we find that, though the table poses $\Delta_{11}\oplus \Delta_{9,1}\oplus {}^3\Delta_5[3]$ for $\mathbf{18}$, $\Delta_{9,3}[3]$ would give the same $T_{(2)}$-eigenvalue, since $(-1280)(4^{-1}+1+4)+\frac{2^{12}-1}{3}=-5355$. However, if $\Delta_{9,3}[3]$ were the correct global Arthur parameter, one could deduce a $T_{(\sqrt{-3})}$-contribution from $\Delta_{9,3}$ that would be incompatible with the $T_{(\sqrt{-3})}$ eigenvalue for $\mathbf{16}$, assuming that $\Delta_{11}\oplus \Delta_{9,3}\oplus{}^3\Delta_6[2]\oplus [2]$ for $\mathbf{16}$ (which is linked to a congruence in Example \ref{Klingen5} of \S \ref{Klingen}) is correct.
 \end{Example}
 \begin{Example} $\mathbf{a=4,b=2}$. This time $\dim(M(V,K_L))=2$, and $\tr(T_{(2)})$ is $2628$. For either of the two cuspidal automorphic representations $\tilde{\pi}$ of $\GL_4(\MA_{E})$ (coming from two Hecke eigenvectors) we write $\Delta^{(2)}_{11,5}$.
 If $\Delta_{11,5}^{(2)}\oplus{}^3\Delta_8[2]\oplus [4]$ is correct for $\mathbf{12}$ and $\mathbf{13}$ then
 $$4^{11/2}\tr(t_{\FP}(\tilde{\pi}))+4((6\sqrt{-14})^2+2\cdot 2^8)(1+4)+2^8\left(\frac{4^4-1}{3}\right)+\frac{2^{12}-1}{3}=23319\pm 162\sqrt{193},$$
 which would imply that $4^{11/2}\tr(t_{\FP}(\tilde{\pi}))=34\pm 162\sqrt{193}$, then that the eigenvalues of $T_{(2)}$ on $M(V,K_L)$ are
$$4^{11/2}\tr(t_{\FP}(\tilde{\pi}))+2^8(2^3-2^2+2-1)=1314\pm 162\sqrt{193}.$$
This is consistent with $\tr(T_{(2)})=2628$ on $M(V,K_L)$, which was confirmed independently by the method of \cite{DumSimple}.
 \end{Example}
 \subsubsection{$\FP=(\sqrt{-3})$}\label{FP3}
 The formulas for $\lambda_i(T_{\FP})$ given at the end of \S \ref{eigenvectors} do not apply to the ramified prime $\FP=(\sqrt{-3})$.
 The following seems to work, though we have not justified it.
 $$\lambda_i(T_{(\sqrt{-3})})=3^{11/2}\tr(t_{\FP}(\tilde{\pi}_i))+3^6-1.$$ We hope to recover the computed Hecke eigenvalues by substituting for $t_{\FP}(\tilde{\pi}_i)$ a Satake parameter $t_{\FP}(A_i)$ associated with the conjectured global Arthur parameter $A_i$.
 We must explain what we mean by this, given that the representations are not always unramified at $\FP=\sqrt{-3}$. We shall not attempt to apply the formula to cases involving any $\Delta_{2a+3,2b+1}$. The automorphic representation $[d]$ of $\GL_d(\MA_E)$ is unramified at $\FP$, and we calculate $t_{\FP}([d])$, an actual Satake parameter, just as before. The automorphic representations $\psi_6$ and $\overline{\psi_6}$ of $\GL_1(\MA_E)$ are unramified at $\FP$, with Satake parameter $(\sqrt{-3})^6=-27$ in both cases.

 For $\Delta_{11}=\tilde{\Pi_{\Delta}}$, where $\Delta$ is the normalised cusp form of weight $12$ for $\SL_2(\MZ)$, since $\Pi_{\Delta}$ is unramified at $3$, the local representation of $W_{\MQ_3}$ (and therefore its restriction to $W_{E_{\FP}}$) is unramified, and we just take $t_{\FP}(\Delta_{11})=t_3(\Pi_{\Delta})$, an actual Satake parameter.

 For ${}^3\Delta_{k-1}=\tilde{\Pi_f}$, where $f\in S_k(\Gamma_0(3),\chi_{-3})$ with $k$ odd, while the local representation of $W_{\MQ_3}$ is ramified, its restriction to $W_{E_{\FP}}$ is unramified, and using a theorem of Langlands and Carayol \cite[Theorem 4.2.7 (3)(a)]{Hida},
 $$3^{(k-1)/2}\,t_{\FP}({}^3\Delta_{k-1})=\diag(a_3(f),3^{k-1}/a_3(f)).$$ For $k=9$ this is $\diag(45-18\sqrt{-14},45+18\sqrt{-14})$, and for $k=11$ it is
 $\diag(-27+108\sqrt{-5},-27-108\sqrt{-5})$.

 For ${}^3\Delta_{k-1}$, with $k$ even and $f$ a newform in $S_k(\Gamma_0(3))$, ${}^3\Delta_{k-1}$ is ramified at $\FP$, but we just try using the same formula, $3^{(k-1)/2}\,t_{\FP}({}^3\Delta_{k-1})=\diag(a_3(f),3^{k-1}/a_3(f))$. For $k=6,8,10,12$ this is $\diag(3^2,3^3)$, $\diag(-3^3,-3^4)$, $\diag(-3^4,-3^5)$ and $\diag(-3^5,-3^6)$ respectively.

 Putative Satake parameters at $\sqrt{-3}$ for the conjectured global Arthur parameters $A_i$ (excluding $i=12,13,16,18$) now follow from those we have attached to their constituent parts.

 \begin{Proposition}\label{RecoverT3} For all $1\leq i\leq 20$, excluding $i=12,13,16,18$ (which involve $\Delta_{2a+3,2b+1}$), the formula $3^{11/2}\tr(t_{(\sqrt{-3})}(A_i))+3^6-1$, applied to the conjectured global Arthur parameter $A_i$ in the last column of the table, recovers the computed Hecke eigenvalue for $T_{(\sqrt{-3})}$ in the third column.
 \end{Proposition}
\begin{proof}

 $\mathbf{1}: [12]$.
 $$266448=\frac{3^{12}-1}{2} + 3^6-1.$$

 $\mathbf{2}: \Delta_{11}\oplus [10]$.
 $$89552=252+3\left(\frac{3^{10}-1}{2}\right)+3^6-1.$$

 $\mathbf{6}: {}^3\Delta_{10}[2]\oplus {}^3\Delta_7\oplus [6]$.
 $$9368=(-27-27)(1+3)+3^2(-3^3-3^4)+3^3\left(\frac{3^6-1}{2}\right)+3^6-1.$$

 $\mathbf{7}: {}^3\Delta_{11}\oplus {}^3\Delta_8[2]\oplus [6]$.
 $$10664=(-3^5-3^6)+3(45+45)(1+3)+3^3\left(\frac{3^6-1}{2}\right)+3^6-1.$$

 $\mathbf{19}: ({}^3\Delta_5\otimes\psi_6)\oplus \psi_6[4]\oplus\overline{\psi}_6[6],\,\,$ $\mathbf{20}: ({}^3\Delta_5\otimes\overline{\psi}_6)\oplus \overline{\psi}_6[4]\oplus\psi_6[6]$.
 $$-13312=(3^2+3^3)(-27)+3(-27)\left(\frac{3^4-1}{2}\right)+(-27)\left(\frac{3^6-1}{2}\right)+3^6-1.$$
 Other examples may be checked similarly.
 \end{proof}
 \subsection{Eisenstein lattices of rank $12$}\label{Eisenstein}
 Recall from Section \ref{compres} the genus of $5$ classes of rank-$12$, $\sqrt{-3}$-modular lattices, for which we obtained the matrix for the Hecke operator $T_{(2)}$. One finds that the eigenvalues match those of $\mathbf{1,2,4,8,9}$, so presumably the associated automorphic representations have the same global Arthur parameters. Among the conjectured global Arthur parameters on the list, these are precisely those that do not involve anything of level $\Gamma_0(3)$, $\psi_6$, $\overline{\psi_6}$ or some $\Delta_{2a+3,2b+1}$. The unitary group in question is isomorphic to the one we already considered (quasi-split at all finite primes), but the open compact subgroups $K_L$ differ locally at $3$.
 \section{Congruences of Hecke eigenvalues}
 Recall the basis $\{v_1,v_2\ldots,v_{20}\}$ for $M(\mathrm{triv},K_L)$, simultaneous eigenvectors for $T_{(2)}$ and $T_{(\sqrt{-3})}$. In $18$ out of the $20$ cases the eigenvalue for $T_{(2)}$ is rational, and in the other two it lies in $\mathbb{Q}(\sqrt{193})$, which has class number $1$. In those cases where the eigenspace for $T_{(2)}$ is one-dimensional (i.e. for $1\leq i\leq 18$), we may, and do, scale each $v_i$ to have algebraic integer values with no common factor, and it is a common eigenvector for all the $T(\gamma)$.
 \begin{Lemma} Consider $v\in M(\mathrm{triv},K_L)$, with values in $\MZ$ (as a function on the $20$-element set $\Sigma_{K_L}$). Suppose that $v=\sum_{i=1}^{20}c_iv_i$ ($v_i$ scaled as above), with $c_i$ in a sufficiently large number field $F$ (in fact $c_i\in\MQ$ for $1\leq i\leq 11$ and $14\leq i\leq 18$, and $c_i\in\MQ(\sqrt{193})$ for $i=12,13$). Suppose that $\fq$ is a prime of $\CO_{F}$ with $\ord_{\fq}(c_i)<0$, for some $1\leq i\leq 18$. Then there exists some $j\neq i$ such that
 $$\lambda_i(T)\equiv\lambda_j(T)\pmod{\fq}\,\,\,\forall T\in\MT,$$
 where $\MT$ is the $\MZ$-subalgebra of $\End(M(\mathrm{triv},K_L))$ generated by all the $T(\gamma)$.
 \end{Lemma}
 \begin{proof} Suppose for a contradiction that this is not the case. Then for each $j\neq i$ there is some $T(j)\in\MT$ such that $\lambda_i(T(j))\not\equiv\lambda_j(T(j))\pmod{\fq}$. Now apply $\prod_{j\neq i}(T(j)-\lambda_j(T(j)))$ to both sides of $v=\sum_{k=1}^{20}c_kv_k$.
 The left-hand-side remains integral. On the right-hand-side, all the terms for $k\neq i$ are killed, whereas $c_iv_i$ is multiplied by $\prod_{j\neq i}(\lambda_i(T(j))-\lambda_j(T(j)))$, which fails to cancel the $\fq$ in the denominator of $c_i$ (hence of at least one of the entries of $c_iv_i$), contradicting the integrality of the left-hand-side.
 \end{proof}
 \begin{Proposition}\label{congruences} The following congruences hold, where ``$\mathbf{i}\equiv\mathbf{j}\pmod{\fq}$'' is shorthand for ``$\lambda_i(T)\equiv\lambda_j(T)\pmod{\fq}$ for all $T\in\MT$'', and in each case we have replaced $F$ by the smaller field (usually $\MQ$) generated by all the $\lambda_i(T)$ and $\lambda_j(T)$ for the particular pair $i,j$.

 \begin{enumerate} \item $\mathbf{2}\equiv\mathbf{1}\pmod{691},\,\,\,\mathbf{4}\equiv\mathbf{1}\pmod{1847},\,\,\,\mathbf{9}\equiv\mathbf{1}\pmod{809},\,\,\,\mathbf{8}\equiv\mathbf{2}\pmod{809},$\newline $\mathbf{7}\equiv\mathbf{3}\pmod{809};$
  \item $\mathbf{3}\equiv\mathbf{1}\pmod{73},\,\,\,\mathbf{5}\equiv\mathbf{2}\pmod{61},\,\,\,\mathbf{6}\equiv\mathbf{4}\pmod{41};$
  \item $\mathbf{3}\equiv\mathbf{2}\pmod{17},\,\,\,\mathbf{7}\equiv\mathbf{8}\pmod{17},\,\,\,\mathbf{15}\equiv\mathbf{14}\pmod{17};$
  \item $\mathbf{16}\equiv\mathbf{11}\pmod{11},\,\,\,\mathbf{7}\equiv\mathbf{12}\pmod{\fq},\,\,\mathbf{7}\equiv\mathbf{13}\pmod{\overline{\fq}},\,\,\text{with $\fq\mid 59$},$\newline
  $\mathbf{9}\equiv\mathbf{12}\pmod{\fq},\,\,\mathbf{9}\equiv\mathbf{13}\pmod{\overline{\fq}},\,\,\text{with $\fq\mid 23$}.$
  \end{enumerate}
  \end{Proposition}
  \begin{proof} We applied the above lemma, with simple $v$ such as $(1,0,\ldots,0)^t$, ignoring the ubiquitous primes $q\leq 7$, but otherwise for any instance of $\ord_{\fq}(c_i)<0$, looking among the $j\neq i$ with $\ord_{\fq}(c_j)<0$ and checking the computed values of $\lambda_k(T_{(2)})$ (and in one case also $\lambda_k(T_{(\sqrt{-3})})$) to see whether $\mathbf{i}\equiv\mathbf{j}\pmod{\fq}$ appears to hold, and to eliminate all the other possible $j$.
  \end{proof}
  Though we cannot prove it, not having the eigenvectors $v_{19},v_{20}$, we believe that also $\mathbf{17}\equiv\mathbf{19,20}\pmod{13}$. The rational prime $13$ divides the denominator of $c_{17}$. Furthermore, if our guess for the Arthur parameters is correct then we have an explanation for this congruence, belonging to the second batch, which we present in \S 5.2. The eigenvalues $\lambda_{19}(T),\lambda_{20}(T)$ should lie in $\MQ(\sqrt{-3})$, for all $T\in\MT$. In this field $(13)=\fq\overline{\fq}$, $\fq=((7+\sqrt{-3})/2)$, but we really do mean mod $13$, i.e. both congruences should be modulo both $\fq$ and $\overline{\fq}$.

  There is no guarantee that we have found all such congruences, but we suspect that we have for $q\geq 11$. See \cite[X.4]{ChenevierLannes} for a more thorough method, applied to the genus of rank-$24$ Niemeier lattices, which finds all the congruences among Hecke eigenvalues, including some modulo higher powers of primes and involving more than two eigenspaces at once. M\'egarban\'e successfully employed a somewhat cruder version of their method \cite[Lemme 4.3.2]{Megarbane}, finding optimal moduli for his congruences. Following him, for a second proof of the congruence $\mathbf{16}\equiv\mathbf{11}\pmod{11}$ we checked that the reductions mod $11$ of $v_{16}$ and $v_{11}$ are scalar multiples of one another, from which the mod $11$ congruence of Hecke eigenvalues follows easily. The greatest common divisor of $(\lambda_{16}(T_{(2)})-\lambda_{11}(T_{(2)}))$ and $(\lambda_{16}(T_{(\sqrt{3})})-\lambda_{11}(T_{(\sqrt{-3})}))$ is $2^3\cdot 3^3\cdot 11$, but mod $2$ and mod $3$, $v_{16}$ is not a scalar multiple of $v_{11}$, so we were unable to use his method to increase the modulus beyond $11$.

  \subsection{Ramanujan-type congruences}
 The first batch of congruences is\newline
 $\mathbf{2}\equiv\mathbf{1}\pmod{691},\,\,\,\mathbf{4}\equiv\mathbf{1}\pmod{1847},\,\,\,\mathbf{9}\equiv\mathbf{1}\pmod{809},\,\,\,\mathbf{8}\equiv\mathbf{2}\pmod{809},$\newline $\mathbf{7}\equiv\mathbf{3}\pmod{809}.$

Recall that $A_i$ is the global Arthur parameter conjectured for $\pi_i$, and for any prime ideal $\FP$ of $\CO_E$, let $A_i(\FP)$ be what the Hecke eigenvalue of $T_{\FP}$ on $v_i$ would be, assuming that $A_i$ is the correct global Arthur parameter. (For $\FP=\sqrt{-3}$, this is admitting our guesses for how to calculate the Hecke eigenvalues, in \S\ref{FP3}.) Recall that we know $A_i(\FP)=\lambda_i(T_{\FP})$ for most $i$, for $\FP=(2)$. Wherever we have proved a congruence $\mathbf{i}\equiv\mathbf{j}\pmod{\fq}$, we would expect to find that $A_i(\FP)\equiv A_j(\FP)\pmod{\fq}$, for all $\FP$, and if we are able to prove this then we may view it as supporting the correctness of $A_i$ and $A_j$. For the first batch we will do so using well-known ``Ramanujan-type'' congruences.

 \begin{Proposition} For all primes ideals $\FP$ of $\CO_E$,
 $A_2(\FP)\equiv A_1(\FP)\pmod{691},\,\,\,A_4(\FP)\equiv A_1(\FP)\pmod{1847},\,\,\, A_9(\FP)\equiv A_1(\FP)\pmod{809},\,\,\, A_8(\FP)\equiv A_2(\FP)\pmod{809}$, and $A_7(\FP)\equiv A_3(\FP)\pmod{809}.$
 \end{Proposition}
 \begin{proof}
  First consider $\mathbf{2}\equiv\mathbf{1}\pmod{691}.$ Recall that $A_1=[12]$ and $A_2=\Delta_{11}\oplus [10]$. At any split prime $(p)=\FP\overline{\FP}$,
 $A_2(\FP)=\tau(p)+(p+p^2+\ldots +p^{10})$ and $A_1(\FP)= 1+p+\ldots +p^{11},$
 so the desired $A_2(\FP)\equiv A_1(\FP)\pmod{691}$ is equivalent to Ramanujan's congruence $\tau(p)\equiv 1+p^{11}\pmod{691}$ (and likewise for $\overline{\FP}$). Ultimately, the $691$ arises as a divisor of the Bernoulli number $B_{12}$, equivalently of $\zeta(1-12)$ or of $\zeta(12)/\pi^{12}$.
 At an inert prime $(p)=\FP$ we get
 $$A_2(\FP)=(\tau(p))^2-2\cdot p^{11}+(p^2+p^4+\ldots +p^{20})$$ and $$A_1(\FP)=1+p^2+\ldots + p^{22},$$ so $A_2(\FP)\equiv A_1(\FP)\pmod{691}$ follows from
 $$(\tau(p))^2\equiv 1+2 p^{11}+p^{22}=(1+p^{11})^2\pmod{691},$$ the square of Ramanujan's congruence. The check for the ramified prime $\FP=\sqrt{-3}$ is similar to that for split primes.

 Now we prepare for the case $\mathbf{4}\equiv\mathbf{1}\pmod{1847},$ first looking at an analogue of Ramanujan's mod $691$ congruence. The space $S_{11}(\Gamma_0(3),\chi_{-3})$ is $2$-dimensional, spanned by a Hecke eigenform $g=$
 $$q+12\sqrt{-5} q^2+(-27+108\sqrt{-5})q^3+304q^4-1272\sqrt{-5}q^5+(-6480-324\sqrt{-5})q^6+17324 q^7+\ldots$$
 and its Galois conjugate. It follows, from the fact that for $p\neq 3$ the adjoint $T_p^*=\langle p\rangle T_p$, that $a_p(g)$ is real (hence rational) for split (in $E$) $p$, purely imaginary (hence with rational square) for inert $p$. The prime $1847$ is a divisor of the generalised Bernoulli number $B_{11,\chi_{-3}}$, equivalently of $L(1-11,\chi_{-3})$ or of $L(11,\chi_{-3})/(\sqrt{3}\pi^{11})$. There is a congruence between the Hecke eigenvalues of $g$ and an Eisenstein series $E_{11}^{1,\chi_{-3}}\in M_{11}(\Gamma_0(3),\chi_{-3})$:
 $$a_p(g)\equiv 1+\chi_{-3}(p)p^{10}\pmod{\fq},$$
 with $(1847)=\fq\overline{\fq}$ in $\MQ(\sqrt{-5})$. This is for {\em all} primes $p$. A proof of this kind of generalised Ramanujan-style congruence, presumably well-known, is recorded in \cite[Proposition 2.1]{DumPacific}. As a consequence, for $p$ split in $E$ we have
 $$a_p(g)\equiv 1+p^{10}\pmod{1847},$$ while for $p$ inert in $E$ we have
 $$a_p(g)^2\equiv (1-p^{10})^2\pmod{1847},$$ and for $p=3$, the ramified prime, $a_3(g)\equiv 1\pmod{1847}$.

 Now recall that $A_1=[12]$ and $A_4={}^3\Delta_{10}[2]\oplus [8]$. At a split prime $p$ the desired $A_4(\FP)\equiv A_1(\FP)\pmod{1847}$ would be
 $$(a_p(g))(1+p)+(p^2+p^3+\ldots +p^9)\equiv 1+p+\ldots +p^{11}\pmod{1847},$$
 which, after cancellation, is equivalent to $(1+p)$ times the known $a_p(g)\equiv 1+p^{10}\pmod{1847},$ while at an inert $p$ the desired
 $$(a_p(g)^2+2 p^{10})(1+p^2)+(p^4+p^6+\ldots +p^{18})\equiv 1+p^2+\ldots +p^{22}\pmod{1847},$$
 is equivalent to $(1+p^2)$ times the known $a_p(g)^2\equiv (1-p^{10})^2\pmod{1847}.$ Finally, for $p=3$, bearing in mind \S\ref{FP3}, in the split prime calculation, for $a_p(g)$ we should substitute the trace of $3^{5}\,t_{\FP}({}^3\Delta_{10})=\diag(a_3(g),3^{10}/a_3(g)),$ which will produce the same $1+3^{10}\pmod{1847}$, given the known $a_3(g)\equiv 1\pmod{1847}$.

 Similarly, $809$ divides $L(1-9,\chi_{-3})$, and the remaining congruences in this first batch may be accounted for by a congruence between a cusp form and an Eisenstein series in $M_{9}(\Gamma_0(3),\chi_{-3})$.
\end{proof}
 One might imagine that filling in the last column of the table was a matter of making lots of guesses $A$ giving the right $c_{\infty}:z\mapsto \diag\left((z/\overline{z})^{11/2},(z/\overline{z})^{9/2},\ldots,(z/\overline{z})^{-11/2}\right),$ for each one computing $4^{11/2}\tr(t_{(2)}(A))+\frac{2^{12}-1}{3}$, and hoping to match all the $\lambda_i(T_{(2)})$ in the second column, which had been computed using neighbours. This is only partly true.
 It was easy to guess that $\mathbf{1}$, which has the largest $T_{(2)}$ eigenvalue, should have global Arthur parameter $[12]$, and this did produce the correct $\lambda_1(T_{(2)})$.
 But we then found $A_2, A_4, A_8$ and $A_9$ using the congruences as a guide. The congruence mod $691$, which we recognised as the modulus of Ramanujan's congruence, suggested trying $\Delta_{11}\oplus [10]$ for $A_2$, and when we did, it recovered the $\lambda_2(T_{(2)})$ correctly. Having similarly seen $809$ and $1847$ before, and already having guesses $A_1$ and $A_2$, the congruence $\mathbf{9}\equiv\mathbf{1}\pmod{809}$ led directly to the guess $A_9$, $\mathbf{8}\equiv\mathbf{2}\pmod{809}$ to $A_8$ and $\mathbf{4}\equiv\mathbf{1}\pmod{1847}$ to $A_4$.
 \subsection{Ramanujan-type congruences of local origin}
 The next batch of congruences is\newline
 $\mathbf{3}\equiv\mathbf{1}\pmod{73},\,\,\,\mathbf{5}\equiv\mathbf{2}\pmod{61},\,\,\,\mathbf{6}\equiv\mathbf{4}\pmod{41}$ (proved) and
 $\mathbf{17}\equiv\mathbf{19,20}\pmod{13}$ (believed).

\begin{Proposition} For all prime ideals $\FP$ of $\CO_E$,
$A_3(\FP)\equiv A_1(\FP)\pmod{73},\,\,\,A_5(\FP)\equiv A_2(\FP)\pmod{61},\,\,\,A_6(\FP)\equiv A_4(\FP)\pmod{41}$ and
 $A_{17}(\FP)\equiv A_{19}(\FP),A_{20}(\FP)\pmod{13}$.
\end{Proposition}
\begin{proof}
 The space $S_{12}(\Gamma_0(3))$ is spanned by a Hecke eigenform $f=q+78q^2-243q^3+\ldots$. For all primes $p\neq 3$ there is a congruence
 $$a_p(f)\equiv 1+p^{11}\pmod{73},$$
 which is of the same shape as Ramanujan's, but whereas $\Delta$ has level $1$ (same as $E_{12}$), $f$ has level $3$. The modulus arises as a divisor of $3^{12}-1$, and may be viewed as a divisor of $\zeta_{\{3\}}(12)/\pi^{12}$, where $\zeta_{\{3\}}(s)=(1-3^{-s})\zeta(s)$, the Riemann zeta function with the Euler factor at $3$ omitted. Such congruences ``of local origin'' were anticipated by Harder in \cite[\S 2.9]{HarderSecOps}, and proved in \cite[Theorem 1.1]{DumFret} or \cite[Theorem 1]{BillereyMenares}.

 Recalling $A_1=[12]$ and $A_3={}^3\Delta_{11}\oplus [10]$, this congruence of local origin may be used to prove $A_3(\FP)\equiv A_1(\FP)\pmod{73}$ in exactly the same way as Ramanujan's congruence was used to prove $A_2(\FP)\equiv A_1(\FP)\pmod{691}$ in the previous subsection, at least for split and inert $p$. For $p=3$ (i.e. $\FP=(\sqrt{-3})$), note that $q\mid(3^k-1)\implies q\mid(3^{k/2}-1)(3^{k/2}+1)$. A theorem of Gaba and Popa \cite[Theorem 1]{GabaPopa} shows that if $q\mid(3^{k/2}+\epsilon)$, with $\epsilon=\pm 1$, then $f$ is in the $\epsilon$ eigenspace for the Atkin-Lehner involution $W_3$. (In this particular example, it happens that $\epsilon=1$.) Then $a_3(f)=-\epsilon 3^{(k/2)-1}$. In the split prime calculation, for $a_p(f)$ we substitute the trace of $\diag(a_3(f),3^{k-1}/a_3(f))$, which is $-\epsilon(3^{(k/2)-1}+3^{k/2})$. Since $q\mid(3^{k/2}+\epsilon)$ (with $q=73$), we may substitute $3^{k/2}$ for $-\epsilon\pmod{73}$, obtaining $-\epsilon(3^{(k/2)-1}+3^{k/2})\equiv 3^{k-1}+3^k\equiv 1+3^{k-1}\pmod{73}$, from which we may proceed as in the split case.

 The congruences $\mathbf{5}\equiv\mathbf{2}\pmod{61}$ and $\mathbf{6}\equiv\mathbf{4}\pmod{41}$ may similarly be dealt with using Ramanujan-type congruences of local origin at the prime $3$, with $61\mid (3^{10}-1)$ (the example from \cite[\S 2.9]{HarderSecOps}) and $41\mid(3^8-1)$. Note that $S_{10}(\Gamma_0(3))$ is $2$-dimensional, but the Hecke eigenform $q-36q^2-81q^3\ldots$ is the one participating in the congruence.

 Finally, since $13\mid (3^6-1)$, there is a congruence mod $13$ of local origin, between the Hecke eigenvalues of $E_6$ and those of the cusp form spanning $S_6(\Gamma_0(3))$. Formally into $A_{19}=({}^3\Delta_5\otimes\psi_6)\oplus \psi_6[4]\oplus\overline{\psi}_6[6]$, we substitute ${}^3\Delta_5\equiv [6]-[4]\pmod{13}$ to get $\psi_6[6]\oplus\overline{\psi_6}[6]$, which is the same as ${}^3\Delta_6[6]$, i.e. $A_{17}$. Likewise for $A_{20}$.
\end{proof}
In fact, recognition of the modulus in the congruence $\mathbf{3}\equiv\mathbf{1}\pmod{73}$ led to the guess $\mathbf{3}: {}^3\Delta_{11}\oplus [10]$, which then produced the correct $\lambda_3(T_{(2)})$. Combined with the last congruence of the previous subsection, this then allowed us to guess the global Arthur parameter for $\mathbf{7}$ too.

 \subsection{Level-raising congruences}
 The next batch of congruences is\newline
 $\mathbf{3}\equiv\mathbf{2}\pmod{17},\,\,\,\mathbf{7}\equiv\mathbf{8}\pmod{17},\,\,\,\mathbf{15}\equiv\mathbf{14}\pmod{17}.$

  Since $a_3(\Delta)=252\equiv -(3^5+3^6)\pmod{17},$ $\Delta$ satisfies the criterion $$a_p(\Delta)\equiv \pm (p^{(k/2)-1}+p^{(k/2)})\pmod{\ell}$$ (with $k=12$, $p=3$ and $\ell=17$) for raising the level by $p$, i.e. there exists a newform $f\in S_{12}(\Gamma_0(3))$ such that
  $$a_q(f)\equiv a_q(\Delta)\pmod{17}\,\,\,\forall q\neq 3.$$
  This raising of the level is a theorem of Ribet \cite{Ribet} in the case $k=2$, completed by Diamond in general for $k\geq 2$ \cite{Diamond}.
  In other words, $\Delta$ and $f$ share the same residual Galois representation at $\ell=17$. (Note that the conditions that this should be irreducible, and that $\ell>k+1$, $\ell\neq p$ are satisfied.)

  On the basis of Ramanujan-style congruences we had already guessed $\mathbf{2}: \Delta_{11}\oplus [10]$, $\mathbf{3}: {}^3\Delta_{11}\oplus [10]$, $\mathbf{7}: {}^3\Delta_{11}\oplus {}^3\Delta_8[2]\oplus [6]$ and $\mathbf{8}:\Delta_{11}\oplus {}^3\Delta_8[2]\oplus [6]$. The congruences $\mathbf{3}\equiv\mathbf{2}\pmod{17}$ and $\mathbf{7}\equiv\mathbf{8}\pmod{17}$ are now perfectly accounted for by the above level-raising congruence, providing further evidence for the conjectured Arthur parameters. Note that in the argument for $a_p(\Delta)\equiv \pm (p^{(k/2)-1}+p^{(k/2)})\pmod{\ell}$ as a necessary condition for level-raising, the $\pm$ is $-\epsilon$, where $\epsilon$ is the eigenvalue for $W_p$ on $f$, so $a_3(\Delta)\equiv a_3(f)+3^{11}/a_3(f)\pmod{17}$, and we find that $A_i(\FP)\equiv A_j(\FP)\pmod{17}$ even for $\FP=\sqrt{-3}$.

  The congruence $\mathbf{15}\equiv\mathbf{14}\pmod{17}$ now suggests the involvement of $\Delta_{11}$ and ${}^3\Delta_{11}$ in $\mathbf{14}$ and $\mathbf{15}$, and a bit of guesswork aimed at filling in the gaps in $c_{\infty}(\tilde{\pi}_i)$ led to proposals that produced the correct $\lambda_i(T_{(2)})$. Congruences $\mathbf{2}\equiv\mathbf{14}\pmod{17}$ and $\mathbf{3}\equiv\mathbf{15}\pmod{17}$ appear to hold if we look just at the $T_{(2)}$-eigenvalues, but the $T_{(\sqrt{-3})}$-eigenvalues rule them out.

  The last batch of congruences appearing in Proposition \ref{congruences} will be dealt with in the following section. Rather than proving the corresponding $A_i(\FP)\equiv A_j(\FP)\pmod{\fq}$  from known congruences, we will use them to provide evidence for some congruences that are only conjectured.
 \section{Eisenstein congruences for $\MU(2,2)$}
 In \cite{BergDum}, a general conjecture was made on congruences of Hecke eigenvalues, between cuspidal automorphic representations of split reductive groups $G$ and representations parabolically induced from Levi subgroups $M$ of maximal parabolic subgroups $P$, modulo divisors of critical values of $L$-functions associated with the latter. In the case $G=\GL_2$ with $P$ a Borel subgroup, $M\simeq \GL_1\times\GL_1$, it predicts the known Ramanujan-type congruences we have already met (including those of local origin). In the case $G=\GSp_2$, $P$ the Siegel parabolic, $M\simeq \GL_2\times\GL_1$, one recovers a conjecture of Harder \cite{Harder123}. With some small modifications one can relax the split condition, and for unitary groups this is explained in \cite{DumUnitary}.

 For $n\geq 1$ let $\MU(n,n)$ be the linear algebraic group over $\MQ$ whose group of $A$-rational points is given by
\begin{equation}
 \MU(n,n)(A)=\{g \in \GL_n(A \otimes_\MQ E)~|~g^\dagger J_n g = J_n \}
\end{equation}
for any commutative $\MQ$-algebra $A$, where $J_n=\begin{bmatrix} 0_n & -I_n\\I_n & 0_n\end{bmatrix}$. This is the
unitary group associated to an Hermitian form of signature $(n,n)$, since $\sqrt{-3} J_n$ is an Hermitian matrix.
In the case $n=2$, there are two classes of maximal parabolic subgroups. There is the Siegel parabolic, with Levi subgroup $M\simeq \GL_{2,E}$, with
$g\mapsto \begin{bmatrix} g & 0\\0 & \,(g^\dagger)^{-1}\end{bmatrix}$, and the Klingen parabolic, with Levi subgroup $M\simeq \GL_{1,E}\times \MU(1,1)$, with $\left(e,\begin{bmatrix} a & b\\c & d\end{bmatrix}\right)\mapsto \begin{bmatrix} e & 0 & 0 & 0\\0 & a & 0 & b\\0 & 0 & (e^\dagger)^{-1} & 0 \\0 & c & 0 & d\end{bmatrix}$.

 \subsection{Klingen parabolic}\label{Klingen}
 In what follows, the correct scaling of the Deligne period $\Omega^{\pm}$ is as in \cite[\S 4]{DumUnitary}. In our examples, where $\fq$ is not a prime of congruence for $f$ in $S_{k'}(\Gamma_0(3))$, the scaling used by the Magma command LRatio \cite{Magma} is correct up to $\fq$-units.
\begin{Conjecture} Let $f\in S_{k'}(\Gamma_0(3))$ be a normalised Hecke eigenform.
 Suppose that $$\ord_{\fq}\left(\frac{L_{\{3\}}(f,(k'/2)+b+1)}{(2\pi i)^{(k'/2)+b+1}\Omega^{(-1)^{(k'/2)+b+1}}}\right)>0$$ or $$\ord_{\fq}\left(\frac{L_{\{3\}}(f,\chi_{-3},(k'/2)+b+1)}{i\sqrt{3}(2\pi i)^{(k'/2)+b+1}\Omega^{(-1)^{(k'/2)+b}}}\right)>0,$$ where $\fq$ divides a rational prime $q>k'$, and $0<b<(k'/2)-1$. Then, letting $a=(k'-4)/2$, there exists a Hecke eigenform $v\in M(V_{a,b},K_L)$ (notation as in \S \ref{algU4}) such that
 $$\lambda_v(T_{\FP})\equiv \begin{cases} a_p(f)+p^{a+b+2}+p^{a-b+1} & \pmod{\fq} \,\, \text{ if $(p)=\FP\overline{\FP}$};\\
 (a_p(f)^2-2p^{k'-1})+p^{k'-4}(p^3-p^2+p-1) & \pmod{\fq} \,\, \text{ if $(p)=\FP$}.\end{cases}$$
 \end{Conjecture}
This conjecture is less general than that stated in \cite[\S 8]{DumUnitary}. To stick to what is narrowly applicable to the situation here, we have put $E=\MQ(\sqrt{-3})$ rather than a more general quadratic field, and restricted to $f$ of level $\Gamma_0(3)$. This is not necessary, if we simply modify the set $\Sigma$ of ``bad'' primes excluded from the Euler product. The general conjecture asserts the existence of a cuspidal automorphic representation $\tilde{\Pi}$ of $\MU(2,2)$ with Hecke eigenvalues congruent mod $\fq$ to those of an induced representation coming from the base-change to $\GL_2(\MA_E)$ of $\pi_f$. This induced representation depends on a real parameter $s$, which in our case is $b+(1/2)$. The right hand side of the congruence is the Hecke eigenvalue for this induced representation. The conjecture in \cite{DumUnitary} says just that $\tilde{\Pi}$ has set of ramified primes no bigger than $\Sigma$. We have gone a little further, in assuming that $\tilde{\Pi}$ has the same global Arthur parameter as some automorphic representation of $\MU_4(\MA_{\MQ})$ with a $K_L$-fixed vector.
\begin{Example} Let $f\in S_{12}(\Gamma_0(3))$ be the unique normalised newform $f=q+78q^2-243q^3\ldots$. We have $k'=12, a=4$.
Let $b=2$, so $(k'/2)+b+1=9$. Using Magma, $\mathrm{LRatio}(f_{\chi_{-3}},9)=59$. (Here the twisted form $f_{\chi_{-3}}$ is plugged into the Magma command to obtain a suitable normalisation of the twisted $L$-value $L(f,\chi_{-3},9)$.) The conjecture predicts a congruence of Hecke eigenvalues involving one of the $\Delta^{(2)}_{11,5}$, modulo a divisor $\fq$ of $59$ (and consequently for the other one modulo $\overline{\fq}$). Note that $2a+3=11, 2b+1=5$. This congruence would account for $\mathbf{7}\equiv\mathbf{12}\pmod{\fq}$ and $\mathbf{7}\equiv\mathbf{13}\pmod{\overline{\fq}}$, which therefore lend support to this instance of the above conjecture. Recall the guesses $\mathbf{7}: {}^3\Delta_{11}\oplus {}^3\Delta_8[2]\oplus [6]$ and $\mathbf{12}: \Delta^{(2)}_{11,5}\oplus {}^3\Delta_8[2]\oplus [4]$. Note that $p^{a+b+2}+p^{a-b+1}=p^8+p^3=p^3(p^5+1)$, matching perfectly what is left over from the cancellation between $[6]$ and $[4]$.
\end{Example}
\begin{Example}\label{Klingen5} Let $f=q-36q^2-81q^3\ldots$, one of the normalised newforms in $S_{10}(\Gamma_0(3))$. We have $k'=10, a=3$. Let $b=1$, so $(k'/2)+b+1=7$. Using Magma, $\mathrm{LRatio}(f_{\chi_{-3}},7)=22$. The conjecture predicts a congruence mod $11$ of Hecke eigenvalues, involving  $\Delta_{9,3}$. Note that $2a+3=9, 2b+1=3$. This congruence would account for $\mathbf{11}\equiv\mathbf{16}\pmod{11}$, which therefore lends support to the above instance of the conjecture. Recall the guesses $\mathbf{11}: \Delta_{11}\oplus {}^3\Delta_9\oplus {}^3\Delta_6[2]\oplus [4]$ and $\mathbf{16}: \Delta_{11}\oplus\Delta_{9,3}\oplus {}^3\Delta_6[2]\oplus [2]$. Note that $p^{a+b+2}+p^{a-b+1}=p^6+p^3=p^3(p^3+1)$, matching perfectly what is left over from the cancellation between $[4]$ and $[2]$. Note also that the condition $q>k'$ only narrowly holds here, with $11>10$.
\end{Example}

\subsection{Siegel parabolic}
 \begin{Conjecture} Let $f\in S_{k'}(\Gamma_0(3),\chi_{-3})$ (odd $k'>1$) be a normalised Hecke eigenform. Suppose that
 $$\ord_{\fq}\left(\frac{L_{\{3\}}(\Sym^2 f,k'+s)}{(2\pi i)^{k'+2s+1}\langle f,f\rangle}\right)>0,$$ where $\fq$ divides a rational prime $q>2k'$, and $s$ is odd with $1<s\leq k'-2$. Let $a=(k'-4+s)/2$ and $b=(k'-2-s)/2$, so $k'=a+b+3, s=a-b+1$ and $k'+s=2a+4$. Then there exists a Hecke eigenform $v\in M(V_{a,b},K_L)$ such that $\pmod{\fq}$
 $$\lambda_v(T_{\FP})\equiv \begin{cases} a_p(f)(1+p^s) &  \text{ if $(p)=\FP\overline{\FP}$};\\
 (a_p(f)^2+2p^{k'-1})(1+p^{2s})+p^{k'+s-4}(p^3-p^2+p-1) &  \text{ if $(p)=\FP$}.\end{cases}$$
 \end{Conjecture}
 Similar remarks apply, concerning the relation of this to the conjecture in \cite[\S 7]{DumUnitary}, as in the previous subsection. The use of
 $(2\pi i)^{k'+2s+1}\langle f,f\rangle$ for the Deligne period is correct (up to $\fq$-units) for our examples, where $\fq$ is not a prime of congruence for $f$ in $S_{k'}(\Gamma_0(3),\chi_{-3})$.
 \begin{Example} Let $f$ be the Hecke eigenform $q+6\sqrt{-14}q^2+(45-18\sqrt{-14})q^3\ldots$ which, with its Galois conjugate, spans $S_9(\Gamma_0(3),\chi_{-3})$. We have $k'=9$. Take $s=3$, so $k'+s=12$, $a=4,b=2$, $2a+1=11, 2b+1=5$. The Euler factor at $3$ in $L(\Sym^2 f,t)$, which is missing in $L_{\{3\}}(\Sym^2 f,t)$, is $P(3^{-t})^{-1}$, where $P(X)=\det(I-\Sym^2\rho_f(\mathrm{Frob}_3^{-1})|(\Sym^2 V)^{I_3})$, where $\rho_f$ is a $\lambda$-adic Galois representation attached to $f$, on a $2$-dimensional space $V$, and we are taking invariants for an inertia subgroup at $3$. According to a theorem of Langlands and Carayol, for which a convenient reference is \cite[Theorem 4.2.7 (3)(a)]{Hida}, the action of $I_3$ on $V$ is diagonalisable, with the trivial character and the character of order $2$ appearing. On the unramified part, $\mathrm{Frob}_3^{-1}$ acts by the $U_3$ eigenvalue, which is the coefficient of $q^3$, and $\det\rho_f$ is the product of the $(k'-1)$ power of the cyclotomic character and the Galois character associated to $\chi_{-3}$. It follows that
 $$P(X)=1+5022X+43046721X^2,$$
 noting that $43046721=3^{16}$, $-5022=(45-18\sqrt{-14})^2+(45+18\sqrt{-14})^2$ and $(45-18\sqrt{-14})(45+18\sqrt{-14})=3^8$.
 Now $P(3^{-12})=\frac{2^523}{3^6}$, and $23>2k'=18$, so the conjecture predicts a congruence mod $\fq$ of Hecke eigenvalues, involving  $\Delta^{(2)}_{11,5}$, where $\fq\mid 23$. This congruence would account for $\mathbf{9}\equiv\mathbf{12}\pmod{\fq}$ and $\mathbf{9}\equiv\mathbf{13}\pmod{\overline{\fq}}$, which therefore lend support to this instance of the above conjecture. Recall the guesses $\mathbf{9}: {}^3\Delta_8[4]\oplus [4]$ and $\mathbf{12}: \Delta^{(2)}_{11,5}\oplus {}^3\Delta_8[2]\oplus [4]$. The $a_p(f)(1+p^3)$ is exactly what is left after cancellation between $\Delta_8[4]$ and $\Delta_8[2]$.
 \end{Example}
 \begin{Remark} In \cite{DumSimple} (repeated in \cite{DumUnitary}), we looked at the case $k'=9, s=5$, with $q=19$ or $37$, gathering a scrap of evidence for the conjectured congruences by computing the Hecke eigenvalues for $T_{(2)}$ on the $2$-dimensional space $M(V_{5,1},K_L)$.
 Contrary to what was stated there, due to a misunderstanding about the Euler factor at $3$ in the $L$-value computed by the formula, the case $q=19$ actually comes from the missing Euler factor at $3$ rather than the complete $L$-value:
 $$P(3^{-14})=\frac{2^5.5^3.7.19}{3^{12}}\,.$$
 \end{Remark}
 \begin{Remark}\label{technical}
 The motivation for this work was to prove instances of Eisenstein congruences for $\MU(2,2)$, following the work of Chenevier and Lannes on Harder's conjecture. However, as noted in the introduction, because there is now a ``bad'' prime $3$, it appears that it is not yet technically feasible to do something similar here. One of the alternative methods they employed was to use Arthur's multiplicity formula to prove the occurrence of the endoscopic types listed in their paper, and though work of Mok and of Kaletha et. al. provides such a formula in our case \cite{Mok,KMSW}, it appears that not enough is currently known about representations of ramified unitary groups to compute the terms in the formula. Similarly, there are problems in trying to imitate Chenevier and Lannes's use of explicit formulas of analytic number theory, to limit the possible components in the endoscopic decompositions. We are grateful to Chenevier for his comments on this. We thank him also for pointing out that it may be possible to prove, following what Ikeda did in the Niemeier lattices setting \cite{IkedaMiyawaki}, some of our guesses for global Arthur parameters (see \S \ref{thetaseries} below). However, this would not cover those cases involved in the congruences for $\MU(2,2)$ which we wish to prove.
 \end{Remark}
 \section{Hermitian theta series}\label{thetaseries}
 Let $$\CH_m=\{Z\in M_n(\MC)\,\mid\,i(Z^\dagger -Z)>0\}$$ be the Hermitian upper half space of degree $m$.
 Given an integral Hermitian lattice $L\subseteq V_{12}$, we define its Hermitian theta series of degree $m$, $\Ct^{(m)}(L):\CH_m\rightarrow\MC$ by
 $$\Ct^{(m)}(L,Z):=\sum_{\mathbf{x}\in L^m}\exp(\pi i \tr(\langle \mathbf{x},\mathbf{x}\rangle Z)).$$
 Applying \cite[Lemma 2.1]{SchSchSch}, we find that if $L$ is unimodular (e.g. $L=\CO_E^{12}$) then $\Ct^{(m)}(L)$ is a modular form of weight $12$ for the group
 $\tilde{\Gamma}^{(m)}:=$ $$\left\{\begin{pmatrix} A & B\\C & D\end{pmatrix}\in \MU(m,m)(\MQ)\,\mid\, A,D\in M_m(\CO_E), B\in (\sqrt{-3})^{-1}M_m(\CO_E),C\in 3\sqrt{-3}M_m(\CO_E)\right\},$$
 i.e. $$\Ct^{(m)}(L,(AZ+B)(CZ+D)^{-1})=|CZ+D|^{12}\,\Ct^{(m)}(L,Z)\,\,\,\,\,\,\,\forall \begin{pmatrix} A & B\\C & D\end{pmatrix}\in\tilde{\Gamma}^{(m)}.$$
 Starting from this, we obtain $\Theta^{(m)}: M(\mathrm{triv},K_L)\rightarrow M^{(m)}_{12}(\tilde{\Gamma}^{(m)})$, i.e.
 $$\Theta^{(m)}(\sum x_jF_j)=\sum \frac{x_j}{|\Aut(L_j)|}\,\Ct^{(m)}(L_j), $$ where $F_j$ is the characteristic function of the class of the lattice $L_j$. This $\Theta^{(m)}$ ought to map each eigenvector $v_i$ to a Hecke eigenform.

 Given a Hecke eigenform $f\in S_{2k+1}(\Gamma_0(3),\chi_{-3})$ with $2k<12$, we may apply a theorem of Ikeda \cite{IkedaHermitian} to obtain a Hecke eigenform
 $I^{(12-2k)}(f)\in S_{12}(\Gamma^{(12-2k)})$, where $\Gamma^{(m)}:=\MU(m,m)(\MQ)\cap M_{2m}(\CO_E)$ is the standard Hermitian modular group.
 More generally there is such an Hermitian Ikeda lift $I^{(m)}(f)\in S_{2k+m}(\Gamma^{(m)})$ for any even $m\geq 2$, given by an explicit Fourier expansion. It ``extends'' to an automorphic representation of $\MU(m,m)(\MA_{\MQ})$, with global Arthur parameter $\tilde{\Pi_f}[m]$.

 It is easy to show that if $f(Z)\in M_{12}(\Gamma^{(m)})$ then for any $N\in \MZ_{>0}$, $f(NZ)\in M_{12}(\Gamma_N^{(m)})$, where
 $$\Gamma_N^{(m)}:=\left\{\begin{pmatrix} A & B\\C & D\end{pmatrix}\in \MU(m,m)(\MQ)\,\mid\, A,D\in M_m(\CO_E), B\in N^{-1}M_m(\CO_E),C\in NM_m(\CO_E)\right\}.$$ Then $\tilde{\Gamma}^{(m)}\subseteq\Gamma_{3}^{(m)}$.
 It follows that $I^{(12-2k)}(f)(3\,Z)\in S_{12}(\tilde{\Gamma}^{(12-2k)})$.
 Following \cite[VII,Corollaire 3.4]{ChenevierLannes}, and looking back at the conjectured global Arthur parameters in our table, it is natural to guess something like the following:

 \begin{enumerate}
 \item $\Theta^{(2)}(v_4)=I^{(2)}(f)(3\,Z)$, with $f\in S_{11}(\Gamma_0(3),\chi_{-3})$;
 \item $\Theta^{(4)}(v_9)=I^{(4)}(f)(3\,Z)$, with $f\in S_{9}(\Gamma_0(3),\chi_{-3})$.
  \end{enumerate}

 Further, comparing with \cite[\S 7]{IkedaMiyawaki}, we should expect $\Theta^{(3)}(v_7)$ and $\Theta^{(3)}(v_8)$ to come from some kind of Hermitian Miyawaki lifts.

 Recall from \S \ref{Eisenstein} that the space of scalar-valued algebraic modular forms for the genus of $5$ classes of rank-$12$, $\sqrt{-3}$-modular lattices has a basis of eigenvectors $\{w_1,w_2,w_4,w_8,w_9\}$, with $T_{(2)}$-eigenvalues matching those of $\{v_1,v_2,v_4,v_8,v_9\}$. Their Hermitian theta series $\Theta^{(m)}(w_i)$ lie in $S_{12}(\Gamma^{(m)})$ \cite[Theorem 2.1]{HentschelNebe},\cite{CohenResnikoff} and it was conjectured in \cite[Remark 3(b)]{HentschelKriegNebeClassification} that $\Theta^{(4)}(w_9)$ is an Hermitian Ikeda lift.
 \section*{Acknowledgements}
 We thank G. Chenevier and G. Nebe for helpful communications. We thank also the anonymous referee, for a thorough reading and many helpful comments. We found neighbours, and conducted isometry tests, using Magma code written by Markus Kirschmer, now publicly available at \url{http://www.math.rwth-aachen.de/~Markus.Kirschmer/}.

\bibliographystyle{abbrv}
\bibliography{Hecke2}

\begin{thebibliography}{10}

\bibitem{Abdukhalikov13}
K.~Abdukhalikov.
\newblock Unimodular {H}ermitian lattices in dimension 13.
\newblock {\em J. Algebra}, 272(1):186--190, 2004.

\bibitem{Abdukhalikov1415}
K.~Abdukhalikov and R.~Scharlau.
\newblock Unimodular lattices in dimensions 14 and 15 over the {E}isenstein
  integers.
\newblock {\em Math. Comp.}, 78(265):387--403, 2009.

\bibitem{Arthur}
J.~Arthur.
\newblock {\em The endoscopic classification of representations}, volume~61 of
  {\em American Mathematical Society Colloquium Publications}.
\newblock American Mathematical Society, Providence, RI, 2013.
\newblock Orthogonal and symplectic groups.

\bibitem{NebeBachoc}
C.~Bachoc and G.~Nebe.
\newblock Classification of two genera of {$32$}-dimensional lattices of rank
  {$8$} over the {H}urwitz order.
\newblock {\em Experiment. Math.}, 6(2):151--162, 1997.

\bibitem{BergDum}
J.~Bergstr\"om and N.~Dummigan.
\newblock Eisenstein congruences for split reductive groups.
\newblock {\em Selecta Math. (N.S.)}, 22(3):1073--1115, 2016.

\bibitem{BillereyMenares}
N.~Billerey and R.~Menares.
\newblock On the modularity of reducible {${\mathrm mod}\, l$} {G}alois
  representations.
\newblock {\em Math. Res. Lett.}, 23(1):15--41, 2016.

\bibitem{Magma}
W.~Bosma, J.~Cannon, and C.~Playoust.
\newblock The {M}agma algebra system. {I}. {T}he user language.
\newblock {\em J. Symbolic Comput.}, 24(3-4):235--265, 1997.
\newblock Computational algebra and number theory (London, 1993).

\bibitem{Cartier}
P.~Cartier.
\newblock Representations of {$p$}-adic groups: a survey.
\newblock In {\em Automorphic forms, representations and {$L$}-functions
  ({P}roc. {S}ympos. {P}ure {M}ath., {O}regon {S}tate {U}niv., {C}orvallis,
  {O}re., 1977), {P}art 1}, Proc. Sympos. Pure Math., XXXIII, pages 111--155.
  Amer. Math. Soc., Providence, R.I., 1979.

\bibitem{ChenevierLannes}
G.~{Chenevier} and J.~{Lannes}.
\newblock {Formes automorphes et voisins de Kneser des r\'eseaux de Niemeier}.
\newblock {\em ArXiv e-prints}, Sept. 2014.

\bibitem{CohenNebePlesken}
A.~M. Cohen, G.~Nebe, and W.~Plesken.
\newblock Maximal integral forms of the algebraic group {$G_2$} defined by
  finite subgroups.
\newblock {\em Journal of Number Theory}, 72(2):282--308, 1998.

\bibitem{CohenResnikoff}
D.~M. Cohen and H.~L. Resnikoff.
\newblock Hermitian quadratic forms and {H}ermitian modular forms.
\newblock {\em Pacific J. Math.}, 76(2):329--337, 1978.

\bibitem{Diamond}
F.~Diamond.
\newblock Congruence primes for cusp forms of weight {$k\ge 2$}.
\newblock {\em Ast\'erisque}, (196-197):6, 205--213 (1992), 1991.
\newblock Courbes modulaires et courbes de Shimura (Orsay, 1987/1988).

\bibitem{DumPacific}
N.~Dummigan.
\newblock Eisenstein primes, critical values and global torsion.
\newblock {\em Pacific J. Math.}, 233(2):291--308, 2007.

\bibitem{DumSimple}
N.~Dummigan.
\newblock A simple trace formula for algebraic modular forms.
\newblock {\em Exp. Math.}, 22(2):123--131, 2013.

\bibitem{DumUnitary}
N.~Dummigan.
\newblock Eisenstein congruences for unitary groups.
\newblock \url{http://neil-dummigan.staff.shef.ac.uk/unitarycong5.pdf}, 2015,
  preprint.

\bibitem{DumFret}
N.~Dummigan and D.~Fretwell.
\newblock Ramanujan-style congruences of local origin.
\newblock {\em J. Number Theory}, 143:248--261, 2014.

\bibitem{FeitSomeLattices}
W.~Feit.
\newblock Some lattices over {${\mathbb Q}(\sqrt{ -3})$}.
\newblock {\em J. Algebra}, 52(1):248--263, 1978.

\bibitem{GabaPopa}
R.~Gaba and A.~A. Popa.
\newblock A generalization of {R}amanujan's congruence to modular forms of
  prime level.
\newblock {\em J. Number Theory}, 193:48--73, 2018.

\bibitem{GreenbergVoightLatticeMethods}
M.~Greenberg and J.~Voight.
\newblock Lattice methods for algebraic modular forms on classical groups.
\newblock In {\em Computations with Modular Forms}, pages 147--179. Springer,
  2014.

\bibitem{GrossSatake}
B.~H. Gross.
\newblock On the {S}atake isomorphism.
\newblock In {\em Galois representations in arithmetic algebraic geometry
  ({D}urham, 1996)}, volume 254 of {\em London Math. Soc. Lecture Note Ser.},
  pages 223--237. Cambridge Univ. Press, Cambridge, 1998.

\bibitem{GrossAlgebraicModularForms}
B.~H. Gross.
\newblock Algebraic modular forms.
\newblock {\em Israel J. Math.}, 113:61--93, 1999.

\bibitem{Harder123}
G.~Harder.
\newblock A congruence between a {S}iegel and an elliptic modular form.
\newblock In {\em The 1-2-3 of modular forms}, Universitext, pages 247--262.
  Springer, Berlin, 2008.

\bibitem{HarderSecOps}
G.~Harder.
\newblock Secondary operations in the cohomology of {H}arish-{C}handra modules.
\newblock
  \url{http://www.math.uni-bonn.de/people/harder/Manuscripts/Eisenstein/SecOPs.pdf},
  2013, preprint.

\bibitem{HentschelKriegNebeClassification}
M.~Hentschel, A.~Krieg, and G.~Nebe.
\newblock On the classification of lattices over {$\Bbb Q(\sqrt{-3})$}, which
  are even unimodular {$\Bbb Z$}-lattices.
\newblock {\em Abh. Math. Semin. Univ. Hambg.}, 80(2):183--192, 2010.

\bibitem{HentschelNebe}
M.~Hentschel and G.~Nebe.
\newblock Hermitian modular forms congruent to 1 modulo {$p$}.
\newblock {\em Arch. Math. (Basel)}, 92(3):251--256, 2009.

\bibitem{Hida}
H.~Hida.
\newblock {\em Geometric modular forms and elliptic curves}.
\newblock World Scientific Publishing Co., Inc., River Edge, NJ, 2000.

\bibitem{Hoffmann}
D.~W. Hoffmann.
\newblock On positive definite {H}ermitian forms.
\newblock {\em Manuscripta Math.}, 71(4):399--429, 1991.

\bibitem{IkedaMiyawaki}
T.~Ikeda.
\newblock Pullback of the lifting of elliptic cusp forms and {M}iyawaki's
  conjecture.
\newblock {\em Duke Math. J.}, 131(3):469--497, 2006.

\bibitem{IkedaHermitian}
T.~Ikeda.
\newblock On the lifting of {H}ermitian modular forms.
\newblock {\em Compos. Math.}, 144(5):1107--1154, 2008.

\bibitem{Iyanaga}
K.~Iyanaga.
\newblock Class numbers of definite {H}ermitian forms.
\newblock {\em J. Math. Soc. Japan}, 21:359--374, 1969.

\bibitem{KMSW}
T.~{Kaletha}, A.~{Minguez}, S.~W. {Shin}, and P.-J. {White}.
\newblock {Endoscopic Classification of Representations: Inner Forms of Unitary
  Groups}.
\newblock {\em ArXiv e-prints}, Sept. 2014.

\bibitem{Klosin}
K.~Klosin.
\newblock Congruences among modular forms on {$\mathrm{ U}(2,2)$} and the
  {B}loch-{K}ato conjecture.
\newblock {\em Ann. Inst. Fourier (Grenoble)}, 59(1):81--166, 2009.

\bibitem{Kneser}
M.~Kneser.
\newblock Klassenzahlen definiter quadratischer {F}ormen.
\newblock {\em Arch. Math.}, 8:241--250, 1957.

\bibitem{Megarbane}
T.~M\'{e}garban\'{e}.
\newblock Calcul des op\'{e}rateurs de {H}ecke sur les classes d'isomorphisme
  de r\'{e}seaux pairs de d\'{e}terminant 2 en dimension 23 et 25.
\newblock {\em J. Number Theory}, 186:370--416, 2018.

\bibitem{Mok}
C.~P. Mok.
\newblock Endoscopic classification of representations of quasi-split unitary
  groups.
\newblock {\em Mem. Amer. Math. Soc.}, 235(1108):vi+248, 2015.

\bibitem{NebeVenkov}
G.~Nebe and B.~Venkov.
\newblock On {S}iegel modular forms of weight 12.
\newblock {\em J. Reine Angew. Math.}, 531:49--60, 2001.

\bibitem{PleskenSouvignier}
W.~Plesken and B.~Souvignier.
\newblock Computing isometries of lattices.
\newblock {\em Journal of Symbolic Computation}, 24(3):327--334, 1997.

\bibitem{Ribet}
K.~A. Ribet.
\newblock Congruence relations between modular forms.
\newblock In {\em Proceedings of the {I}nternational {C}ongress of
  {M}athematicians, {V}ol.\ 1, 2 ({W}arsaw, 1983)}, pages 503--514. PWN,
  Warsaw, 1984.

\bibitem{SchSchSch}
R.~Scharlau, A.~Schiemann, and R.~Schulze-Pillot.
\newblock Theta series of modular, extremal, and {H}ermitian lattices.
\newblock In {\em Integral quadratic forms and lattices ({S}eoul, 1998)},
  volume 249 of {\em Contemp. Math.}, pages 221--233. Amer. Math. Soc.,
  Providence, RI, 1999.

\bibitem{Schiemann}
A.~Schiemann.
\newblock Classification of {H}ermitian forms with the neighbour method.
\newblock {\em J. Symbolic Comput.}, 26(4):487--508, 1998.

\bibitem{SchoennenbeckSimultaneous}
S.~Sch\"onnenbeck.
\newblock Simultaneous computation of hecke operators.
\newblock {\em J. Algebra}, 501(1):571--597, 2018.

\bibitem{Tamagawa}
T.~Tamagawa.
\newblock On the {$\zeta $}-functions of a division algebra.
\newblock {\em Ann. of Math. (2)}, 77:387--405, 1963.

\end{thebibliography}

\end{document}